\documentclass[preprint]{elsarticle}

\usepackage{amssymb}
\usepackage{graphicx}
\usepackage{amssymb}
\usepackage{amsmath}
\usepackage{amsthm}
\usepackage{amsfonts}
\usepackage{mathptmx}
\newtheorem{theorem}{Theorem}[section]

\newtheorem{proposition}[theorem]{Proposition}

\newtheorem{lemma}[theorem]{Lemma}

\journal{Journal of Approximation Theory}

\usepackage{caption}
 
\begin{document}

\begin{frontmatter}



\title{Laguerre-Angelesco multiple orthogonal polynomials \\ on an $r$-star}


\author{Marjolein Leurs, Walter Van Assche}

\address{Department of Mathematics, KU Leuven, Celestijnenlaan 200B box 2400, BE-3001 Leuven, Belgium}
\ead{Marjolein.Leurs@kuleuven.be, Walter.VanAssche@kuleuven.be}

\begin{abstract}
We investigate the type I and type II multiple orthogonal polynomials on an $r$-star with weight function $|x|^{\beta}e^{-x^r}$, with $\beta>-1$. Each measure $\mu_j$, for $1\leq j \leq r$, is supported on the semi-infinite interval $[0,\omega^{j-1}\infty)$ with $\omega=e^{2\pi i/r}$. For both the type I and the type II polynomials we give explicit expressions, the coefficients in the recurrence relation, the differential equation and we obtain the asymptotic zero distribution of the polynomials on the diagonal. Also, we give the connection between the Laguerre-Angelesco polynomials and the Jacobi-Angelesco polynomials on an $r$-star.
\end{abstract}


\end{frontmatter}


\section{Introduction}

Multiple orthogonal polynomials are a generalization of orthogonal polynomials where we consider polynomials that are orthogonal with respect to a system of measures $(\mu_1, \mu_2, \ldots ,\mu_r)$ for some $r\geq 1$. Hermite used these polynomials for simultaneous rational approximation of $r$ functions $f_1, f_2, \ldots, f_r$, the so-called Hermite-Pad\'{e} approximation, and used them to prove that $e$ is transcendental. There are various families of multiple orthogonal polynomials, some of which are generalizations of the classical orthogonal polynomials named after Jacobi, Laguerre and Hermite. In this paper we discuss an extension of the classical Laguerre polynomials on $[0,\infty)$. An introduction on multiple orthogonal polynomials can be found in, for example, the work of Aptekarev \cite{Aptekarev}, Ismail \cite[Chapter 23]{Ismail}, and Nikishin and Sorokin \cite[Chapter 4]{Nikishin1991}. 

There are two types of multiple orthogonal polynomials, depending on how their orthogonality is defined. Let $\vec{n}=(n_1, n_2, \ldots,n_r)$ be a multi-index of size $|\vec{n}|=n_1+ n_2 + \cdots +  n_r$ and let $\mu_1,\ldots,\mu_r$ be positive measures for which all the moments exist. The type I multiple orthogonal polynomials are given by a vector $(A_{\vec{n},1},\ldots,A_{\vec{n},r})$ of $r$ polynomials, with $\deg A_{\vec{n},j}=n_j-1$ for all $1 \leq j \leq r$, that satisfy the orthogonality conditions
\[ 		\sum_{j=1}^r \int x^k A_{\vec{n},j}(x)\, d\mu_j(x) = 0, \qquad 0 \leq k \leq |\vec{n}|-2,		\]
and normalization
\[		\sum_{j=1}^r \int x^{|\vec{n}|-1} A_{\vec{n},j}(x)\, d\mu_j(x) = 1.		\]
The type II multiple orthogonal polynomial is the monic polynomial $P_{\vec{n}}$ of degree $|\vec{n}|$ for which the following orthogonality conditions hold
\[  	\int x^k P_{\vec{n}}(x)\, d\mu_j(x) = 0, \qquad 0 \leq k \leq n_j-1,		\]
for all $1\leq j\leq r$. The orthogonality conditions and normalization for both the type I and type II multiple orthogonal polynomials give a linear system of $|\vec{n}|$ equations for the $|\vec{n}|$ unknown coefficients of the polynomials. Moreover, the corresponding matrices for the type I and type II polynomials are each other's transpose. So the existence and uniqueness of the two types are equivalent. If there exists a solution of the linear system and it is unique, then the multi-index $\vec{n}$ is called a normal multi-index. Furthermore, if all the multi-indices are normal then the system of measures $(\mu_1,\ldots,\mu_r)$ is a perfect system.  
One of the well-known perfect systems of measures is an Angelesco system. This system was first introduced by Angelesco \cite{Angelesco1919} and later independently by Nikishin \cite{Nikishin1979}. An Angelesco system is a system of $r$ measures $(\mu_1,\ldots,\mu_r)$ where $\mu_j$ has its support in an interval $\Delta_j$ and the intervals $\Delta_1,\ldots,\Delta_r^*$ are pairwise disjoint or touching.  

For $r=1$, we have the usual orthogonal polynomials. The most well-known orthogonal polynomials are the classical orthogonal polynomials named after Jacobi, Laguerre and Hermite. A lot of information on the classical orthogonal polynomials of Jacobi, Laguerre and Hermite can be found in Szeg\H{o}'s book \cite{Szego}.
Consider the monic Jacobi polynomial of degree $n$ on $[0,1]$ which is orthogonal with weight function $w(x)=x^{\beta}(1-x)^{\alpha}$, and denote it by $P_n^{(\alpha,\beta)}$. Note that one usually considers the Jacobi polynomials for the weight $(1-x)^{\alpha}(1+x)^{\beta}$ on $[-1,1]$, but with a change of variables $x \mapsto 2x-1$ one obtains the polynomials $P_n^{(\alpha,\beta)}$ on $[0,1]$. The monic Laguerre polynomials $L_n^{(\alpha)}$, with $\alpha>-1$, are orthogonal with respect to the weight $w(x)=x^{\alpha}e^{-x}$ on $[0,\infty)$. By direct computation, one can show that the Laguerre polynomials are limiting cases of the Jacobi polynomials
\begin{equation}
\label{JA-LA1}
L_n^{(\beta)}(x) = \lim_{\alpha\rightarrow \infty} \alpha^n P_{n}^{(\alpha,\beta)}(x/\alpha).
\end{equation}
A similar link was obtained by Coussement and Van Assche \cite{VACouss} in case $r=2$ for the type II Jacobi-Angelesco polynomials and the type II Laguerre-Hermite polynomials. The type II Jacobi-Angelesco polynomials on $[-1,0]\cup [0,1]$ are the monic polynomials $P_{n,m}^{(\alpha,\beta)}$ that satisfy the orthogonality conditions
\begin{align*}
\int_{-1}^0 x^k P_{n,m}^{(\alpha,\beta)}(x)|x|^{\beta}(1-x^2)^{\alpha}\,dx =0, \qquad 0 \leq k \leq n-1, \\
\int_0^{1} x^k P_{n,m}^{(\alpha,\beta)}(x)x^{\beta}(1-x^2)^{\alpha}\, dx=0, \qquad 0 \leq k \leq m-1,
\end{align*}
where $\alpha,\beta>-1$. These polynomials were studied in detail by Kalyagin \cite{Kalyagin1} and Kaliaguine and Ronveaux \cite{Kalyagin2}. The type II Laguerre-Hermite polynomials on $(-\infty,0]\cup [0,\infty)$ are the monic polynomials $L_{n,m}^{(\beta)}(x)$ that satisfy
\begin{align*}
\int_{-\infty}^0 x^k L_{n,m}^{(\beta)}(x) |x|^{\beta}e^{-x^2}\, dx =0, \qquad 0 \leq k \leq n-1, \\
\int_{0}^{\infty} x^k L_{n,m}^{(\beta)}(x) x^{\beta}e^{-x^2}\, dx =0, \qquad 0 \leq k \leq m-1.
\end{align*}
These polynomials are a limiting case of the type II Jacobi-Angelesco polynomials as follows
\begin{equation}
\label{eq:JA-LA2}
L_{n,m}^{(\beta)}(x) = \lim_{\alpha\rightarrow\infty} (\sqrt{\alpha})^{n+m} P_{n,m}^{(\alpha,\beta)}\left(\frac{x}{\sqrt{\alpha}}\right).
\end{equation}

In this paper we investigate both the type I and type II Laguerre-Angelesco polynomials that are orthogonal on $r$ semi-infinite intervals in the complex plane that form an $r$-star. Take the positive real line and copy it $r-1$ times by rotating it over angle $2\pi/r$. This gives the $r$-star
$\bigcup_{j=1}^r \Delta_j$, where
\[ \Delta_j = \{ t \omega^{j-1}  : 0 < t < \infty \}, \qquad 1\leq j \leq r, \]
with $\omega = e^{2\pi i /r}$, see Fig. \ref{fig:r-starLaguerre}.

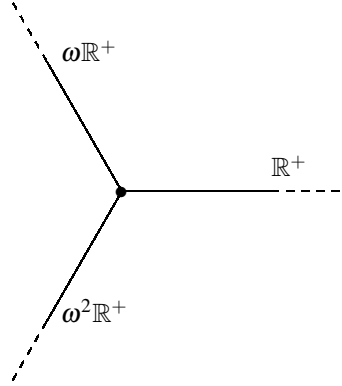
\begin{figure}
\centering
\unitlength=1cm
\begin{picture}(6,6)(-3,-3)
\put(0,0){\line(1,0){2}}
\multiput(2,0)(0.2,0){5}{\line(1,0){0.1}}
\put(0,0){\rotatebox{120}{\put(0,0){\line(1,0){2}}
\multiput(2,0)(0.2,0){5}{\line(1,0){0.1}}}}
\put(0,0){\rotatebox{240}{\put(0,0){\line(1,0){2}}
\multiput(2,0)(0.2,0){5}{\line(1,0){0.1}}}}
\put(0,0){\circle*{0.15}}
\put(2,0.2){$\mathbb{R}^+$}
\put(-0.8,1.732){$\omega\mathbb{R}^+$}
\put(-0.8,-1.732){$\omega^2\mathbb{R}^+$}
\end{picture}
\caption{$r$-star for Laguerre-Angelesco polynomials with $r=3$, $\omega= e^{2\pi i /3}$}
\label{fig:r-starLaguerre}
\end{figure}

%
%
%

For symmetry we take as weight function $w(x) = |x|^{\beta}e^{-x^r}$, $\beta>-1$, and the measure $\mu_j$ is then supported on the ray $\Delta_j$ of the $r$-star with this weight function. Let $\vec{n}=(n_1,n_2,\ldots,n_r)$ be a multi-index of size $|\vec{n}|=n_1+n_2+\cdots+n_r$, then the type I Laguerre-Angelesco polynomials $(A_{\vec{n},1},\ldots,A_{\vec{n},r})$  are defined by the orthogonality conditions
\[  \sum_{j=1}^r \int_{\Delta_j} x^k A_{\vec{n},j}(x;\beta) |x|^{\beta}e^{-x^r}\, dx = 0, \qquad 0 \leq k \leq |\vec{n}|-2,	\]
and normalization 
\[ 	\sum_{j=1}^r \int_{\Delta_j} x^{|\vec{n}|-1} A_{\vec{n},j}(x;\beta) |x|^{\beta}e^{-x^r}\, dx = 1.	\]
The type II Laguerre-Angelesco polynomial $L_{\vec{n}}$ is the monic polynomial of degree $|\vec{n}|$ defined by the orthogonality conditions 
\[	\int_{\Delta_j} x^k L_{\vec{n}}(x;\beta) |x|^{\beta}e^{-x^r}\, dx = 0, \qquad 0 \leq k \leq n_j-1,	\]
for all $1\leq j \leq r$. These polynomials were already studied by Sorokin, see \cite{Sorokin4} and \cite{Sorokin}. He studied the asymptotic behavior of these polynomials and their application in simultaneous Pad\'{e} approximations of functions of Stieltjes type. 

The type I Laguerre-Angelesco polynomials are studied in Section \ref{secIIsubI}, where we give explicit formulas for the polynomials on and near the diagonal, prove their multiple orthogonality, give the recurrence coefficients of the nearest neighbor recurrence relations \cite{VA} and obtain a differential equation of order $r+1$. This differential equation will be used to get the asymptotic distribution of the zeros. The type II Laguerre-Angelesco polynomials are investigated in Section \ref{secIIsubII}. The diagonal type II polynomials were already studied briefly by Sorokin \cite{Sorokin4,Sorokin} for simultaneous Pad\'e approximations. We give explicit formulas for these polynomials, give the recurrence coefficients of the nearest neighbor recurrence relation, obtain a differential equation of order $r+1$ and use this differential equation to get the asymptotic zero distribution. 

Note that there are other possible generalizations of the classical Laguerre polynomials, see \cite{Sorokin3,Sorokin2}. In \cite{Sorokin2}, Sorokin used an AT system instead of an Angelesco system of measures. This is also a perfect system of measures given by $(\mu_1,\ldots,\mu_r)$ where $\mu_j$ has weight function $w_j(x) = x^{\beta_j}e^{-x}$ supported on the positive real line. For a perfect system, the parameters must satisfy the condition: $\beta_i - \beta_j \notin \mathbb{Z}$, $1 \leq i < j \leq r$. These polynomials were also given in \cite{VACouss} where they are given as
multiple Laguerre polynomials of the first kind. For the multiple Laguerre polynomials of the second kind one uses
the weight functions $w_j(x) = x^\beta e^{-c_jx}$, where $c_j >0$ and $c_i \neq c_j$ whenever $i\neq j$.
The asymptotic distribution of the zeros of the multiple Laguerre polynomials of the first kind was given in \cite{NeuschelVA} and appeared already
in \cite{CoussCoussVA} for $r=2$.
The asymptotics of the multiple Laguerre polynomials of the second kind is given in \cite{Lysov,LysovWiel} for $r=2$.  
Another possible extension (also investigated by Sorokin) is to consider $r$ measures supported on the $r$-star, but with weight function $w(x) = x^{\beta} e^{-x}$. Both the extensions were used to investigate Hermite-Pad\'{e} approximation of Stieltjes type functions
\[		f_j = \int \frac{d\mu_j(x)}{x-z}. 		\] 

Similar as in cases $r=1,2$, our Laguerre-Angelesco polynomials on an $r$-star are limiting cases of the Jacobi-Angelesco polynomials on an $r$-star. So we first review some known results of these polynomials from \cite{LeursVA}.

\section{Jacobi-Angelesco polynomials on an $r$-star}  \label{secI}

For the Jacobi-Angelesco polynomials we consider the same configuration of $r$ intervals as for the Laguerre-Angelesco polynomials, but instead of half-lines we take intervals of length $1$. So we have $r$ intervals given by $[0,\omega^{j}]$ with $\omega=e^{2 \pi i /r}$, see Fig. \ref{fig:r-starJacobi}.
%
\begin{figure}
\centering
\unitlength=1cm
\begin{picture}(6,6)(-3,-3)
\put(0,0){\line(1,0){2.5}}
\put(2.5,0){\circle*{0.15}}
\put(0,0){\rotatebox{72}{\put(0,0){\line(1,0){2.5}}
\put(2.5,0){\circle*{0.15}}}}
\put(0,0){\rotatebox{144}{\put(0,0){\line(1,0){2.5}}
\put(2.5,0){\circle*{0.15}}}}
\put(0,0){\rotatebox{216}{\put(0,0){\line(1,0){2.5}}
\put(2.5,0){\circle*{0.15}}}}
\put(0,0){\rotatebox{288}{\put(0,0){\line(1,0){2.5}}
\put(2.5,0){\circle*{0.15}}}}
\put(0,0){\circle*{0.15}}
\put(2.7,0){$1$}
\put(0.95,2.25){$\omega$}
\put(-1.8,1.55){$\omega^2$}
\put(-1.8,-1.60){$\omega^3$}
\put(0.95,-2.35){$\omega^4$}
\end{picture}
\caption{$r$-star for Jacobi-Angelesco polynomials with $r=5$, $\omega= e^{2\pi i /5}$.}
\label{fig:r-starJacobi}
\end{figure}
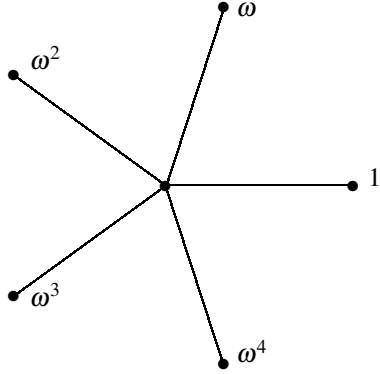
To preserve the symmetry, we take as weight function $w(x) = |x|^{\beta}(1-x^r)^{\alpha}$ and the measure $\mu_j$ supported on the interval $[0,\omega^{j-1}]$, $j=1,\ldots,r$ with this weight function. We only state the definition and explicit expressions of the type I and type II Jacobi-Angelesco polynomials. For details and proofs we refer to \cite{LeursVA}. Note that for $r=1$, we obtain the classical Jacobi polynomials on $[0,1]$.

\subsection{Type I Jacobi-Angelesco polynomials}	\label{secIsubI}
The type I Jacobi-Angelesco polynomials for the multi-index $(n_1,\ldots,n_r)$ on the $r$-star and parameters $\alpha,\beta >-1$ are given by the vector of polynomials $(B_{\vec{n},1},\ldots,B_{\vec{n},r})$, with $\deg B_{\vec{n},j} = n_j-1$ for all $1 \leq j\leq r$ that satisfies 
\[ 	\sum_{j=1}^r \int_0^{\omega^{j-1}} x^k B_{\vec{n},j}(x;\alpha,\beta)|x|^{\beta}(1-x^r)^{\alpha}\, dx = 0, \qquad 0 \leq k \leq |\vec{n}|-2,	\]
and
\[ 	\sum_{j=1}^r \int_0^{\omega^{j-1}} x^{|\vec{n}|-1} B_{\vec{n},j}(x;\alpha,\beta)|x|^{\beta}(1-x^r)^{\alpha}\, dx = 1.	\]
The type I Jacobi-Angelesco polynomials on and near the diagonal on the $r$-star are given explicitly in terms of the polynomials
\begin{equation}    \label{pnJA}
	p_n(x;\alpha,\beta) = 
  \sum_{k=0}^n \binom{n}{k} \frac{\Gamma(n+\alpha+\frac{\beta+k}{r}+1)}{\Gamma(n+\alpha+1) \Gamma(\frac{\beta+k}{r}+1)} (-1)^{n-k}x^k.
\end{equation}
One can show that these polynomials satisfy $n$ orthogonality conditions for the weight $x^{\beta}(1-x^r)^{\alpha}$ on $[0,1]$
\[	\int_0^1 x^{rj-1} p_{n}(x;\alpha,\beta) x^{\beta}(1-x^r)^{\alpha}\, dx = 0, \qquad 1\leq j \leq n,\]
and normalization
\[\int_0^1 x^{rn+r-1} p_{n}(x;\alpha,\beta) x^{\beta}(1-x^r)^{\alpha}\, dx= \frac{n!}{(rn+r\alpha+\beta+r)_{n+1}}.\]
With these conditions for the polynomials $p_n(x;\alpha,\beta)$, one can verify that the following polynomials satisfy the multiple orthogonality conditions and normalization of the type I Jacobi-Angelesco polynomials. First take a multi-index $\vec{n}=(n,\ldots,n)$ and let $\vec{e}_k$ be the $k^{\mbox{\footnotesize th}}$ unit vector of $\mathbb{N}^r$, then the type I Jacobi-Angelesco polynomials on the diagonal are given by
\[ 	B_{\vec{n},j}(x;\alpha,\beta) = \frac{(rn+r\alpha+\beta)_{n}}{r (n-1)!} p_{n-1}(\omega^{-j+1}x;\alpha,\beta), \qquad 1 \leq j\leq r.	\]
Above the diagonal, the type I Jacobi-Angelesco polynomials $B_{\vec{n}+\vec{e}_k,j}$ are given by
\[ 	\tau_{n,r}^{(\alpha,\beta)} B_{\vec{n}+\vec{e}_k,j}(x;\alpha,\beta) = \sum_{\ell=0}^{r-1} \frac{\omega^{\ell(j-k)}}{\nu_n^{\alpha,\beta-\ell}}p_n(x;\alpha,\beta-\ell), \qquad 1 \leq j,k \leq r,	\]
with $\tau_{n,r}$ a normalizing constant and $\nu_n^{\alpha,\beta}$ the leading coefficient of $p_n(x;\alpha,\beta)$.
For every $r>1$ and $n>0$, the type I Jacobi-Angelesco polynomials below the diagonal $B_{\vec{n}-\vec{e}_k,j}$ are given by
\begin{multline*}
\gamma_{n,r}^{(\alpha,\beta)} B_{\vec{n}-\vec{e}_k,j}(x;\alpha,\beta) =  \omega^{j-1} \nu_{n-1}^{(\alpha,\beta)} p_{n-1}(\omega^{-j+1}x;\alpha,\beta-1) \\
- \omega^{k-1} \nu_{n-1}^{(\alpha,\beta-1)} p_{n-1}(\omega^{-j+1}x;\alpha,\beta),
\end{multline*}
\noindent for $1\leq j,k \leq r$ and $\gamma_{n,r}^{(\alpha,\beta)}$ a normalizing constant.

In \cite{LeursVA} one can find the recurrence relations for these polynomials, a differential equation of order $r+1$ and the asymptotic zero behavior for the diagonal polynomials. 

\subsection{Type II Jacobi-Angelesco polynomials} 	\label{secIsubII}
The type II Jacobi-Angelesco polynomial $P_{\vec{n}}^{(\alpha,\beta)}$ for the multi-index $\vec{n}=(n_1,\ldots,n_r)$ on the $r$-star and parameters $\alpha,\beta>-1$ is a monic polynomial of degree $|\vec{n}|$ defined by the orthogonality conditions
\[	\int_0^{\omega^{j-1}}x^k P_{\vec{n}}^{(\alpha,\beta)}(x) |x|^{\beta}(1-x^r)^{\alpha}\, dx = 0, \qquad 0 \leq k \leq n_j-1,	\]
for all $1 \leq j \leq r$. For the multi-index $\vec{n}=(n,n,\ldots,n)$. These polynomials are given by the Rodrigues type formula
\[x^{\beta}(1-x^r)^{\alpha}P_{\vec{n}}^{(\alpha,\beta)}(x) = \frac{(-1)^n}{(rn+r\alpha+\beta+1)_n}\frac{d^n}{dx^n} [x^{\beta+n}(1-x^r)^{\alpha+n}].\]
These polynomials were studied by Kalyagin \cite{Kalyagin1}, who considered the case $r=2$ and investigated the asymptotic behavior of $P_{n,m}$. Later, Kaliaguine and Ronveaux \cite{Kalyagin2} found a third order differential equation, a four term recurrence relation and the asymptotic behavior of the ratio of two neighboring polynomials.

From the Rodrigues formula of $P_{\vec{n}}^{(\alpha,\beta)}$, one can compute the explicit formula
\begin{multline*}
\binom{rn+r\alpha+\beta+n}{n} P_{\vec{n}}^{(\alpha,\beta)}(x) \\
= \sum_{|\vec{k}|=n} \binom{n+\beta}{k_0} \binom{n+\alpha}{k_1} \cdots \binom{n+\alpha}{k_r} x^{n-k_0} (x-1)^{n-k_1} \ldots (x-\omega^{r-1})^{n-k_r}.
\end{multline*}

\section{Laguerre-Angelesco polynomials on the $r$-star}	\label{secII}
Now we can state some new results for the Laguerre-Angelesco polynomials on an $r$-star. For the type I polynomials, we give explicit expressions, the recurrence coefficients for the nearest neighbor recurrence relation and a differential equation for the polynomials on the diagonal. With this differential equation, we can investigate the asymptotic behavior of the zeros of the polynomials on the diagonal. The type II Laguerre-Angelesco polynomials were already studied (briefly) by Sorokin \cite{Sorokin}, but only the Rodrigues formula was stated for the polynomials on the diagonal. From this Rodrigues formula we compute the explicit expression and derive a differential equation of order $r+1$. This differential equation will be used to obtain the asymptotic behavior of the zeros. We also state the Rodrigues formula and an explicit expression for the polynomials above the diagonal. With these expressions, we can compute the recurrence coefficients in the nearest neighbor recurrence relation. Also, in both sections we state the link between the Laguerre-Angelesco polynomials and the Jacobi-Angelesco polynomials on an $r$-star from the previous section.

\subsection{Type I Laguerre-Angelesco polynomials}	\label{secIIsubI}
The type I Laguerre-Angelesco polynomials for the multi-index $(n_1,\ldots,n_r)$ on an $r$-star and parameter $\beta>-1$ are given by the vector of polynomials $(A_{\vec{n},1},\ldots,A_{\vec{n},r})$, which is uniquely defined by
\begin{enumerate}
 \item degree conditions: the degree of $A_{\vec{n},j}$ is $n_j-1$, 
 \item orthogonality condition 
\[		\sum_{j=1}^r \int_0^{\omega^{j-1} \infty} x^k A_{\vec{n},j}(x;\beta) |x|^{\beta}e^{-x^r}\, dx = 0, \qquad 0 \leq k \leq |\vec{n}|-2,	\]
 \item normalization condition
\[ 	\sum_{j=1}^r \int_0^{\omega^{j-1} \infty} x^{|\vec{n}|-1} A_{\vec{n},j}(x;\beta) |x|^{\beta}e^{-x^r}\,dx = 1.	\]
\end{enumerate}

\subsubsection{Explicit expressions}
The polynomials $A_{\vec{n},j}$ on and near the diagonal can be expressed in terms of the polynomials
\begin{equation}
\label{eq:pn}
p_n(x;\beta) = \sum_{k=0}^n \binom{n}{k} \frac{1}{\Gamma(\frac{\beta+k}{r}+1)}(-1)^{n-k}x^k.
\end{equation}
First we prove that these polynomials satisfy $n$ orthogonality conditions and a normalization for the weight $w(x) = x^{\beta}e^{-x^r}$ on $[0,\infty)$. 
\begin{proposition}		\label{proposition1}
The polynomials $p_n(x;\beta)$ given by (\ref{eq:pn}) satisfy 
\begin{equation}	\label{eq:orthopn}
	\int_0^{\infty} x^{rj-1} p_n(x;\beta) x^{\beta}e^{-x^r}\,dx=0, \qquad 1 \leq j \leq n,
\end{equation}
and 
\begin{equation}	\label{eq:normpn}
	\int_0^{\infty} x^{rn+r-1} p_n(x;\beta) x^{\beta}e^{-x^r}\,dx=\frac{n!}{r^{n+1}}.
\end{equation}
\end{proposition}
\begin{proof}
For the orthogonality, we use the explicit expression of $p_n(x;\beta)$ which gives
\[	\int_0^{\infty} x^{rj-1} p_n(x;\beta) x^{\beta}e^{-x^r}dx= \sum_{k=0}^n \binom{n}{k}  \frac{1}{\Gamma(\frac{\beta+k}{r}+1)}(-1)^{n-k}	
   \int_0^{\infty} x^{rj+k+\beta-1} e^{-x^r}\, dx.\]
With the change of variables $x\mapsto x^r$ and the definition of the Gamma function, this equals
\begin{equation}	\label{eq:intpn}
	\frac1r \sum_{k=0}^n \binom{n}{k}(-1)^{n-k}  \frac{\Gamma(\frac{\beta+k}{r}+j)}{ \Gamma(\frac{\beta+k}{r}+1)} 
      =	\frac{1}{r}\sum_{k=0}^n \binom{n}{k}(-1)^{n-k} \left(\frac{\beta+k}{r}+1 \right)_{j-1},
\end{equation}
which is equal to zero by (\ref{eq:lemma1}) of the lemma below. For the normalization, we can use the same expression for the integral as 
(\ref{eq:intpn}), but with $j=n+1$
\[	\int_0^{\infty} x^{rn+r-1} p_n(x;\beta) x^{\beta}e^{-x^r}\,dx = \frac{1}{r}\sum_{k=0}^n \binom{n}{k}(-1)^{n-k} 
  \left(\frac{\beta+k}{r}+1 \right)_{n}.\]
Since $\left(\frac{\beta+k}{r}+1 \right)_{n}$ is a  polynomial in $k$ of degree $n$ with leading term $1/r^n$ we have, due to (\ref{eq:lemma1}) 
and (\ref{eq:lemma2}) from the lemma below, that this expression is equal to 
\[	\frac{1}{r^{n+1}}\sum_{k=0}^n \binom{n}{k}(-1)^{n-k} k^n = \frac{n!}{r^{n+1}}.\]
\end{proof}
In the proof of the previous proposition, we used the following result.

\begin{lemma}	\label{lemma}
For all $n \in \mathbb{N}$ one has
\begin{equation}	\label{eq:lemma1}
	\sum_{k=0}^n \binom{n}{k} (-1)^{n-k} k^m = 0, \qquad 0 \leq m \leq n-1,
\end{equation}
and 
\begin{equation}	\label{eq:lemma2}
	\sum_{k=0}^n \binom{n}{k} (-1)^{n-k} k^n = n!.
\end{equation}
\end{lemma}
\begin{proof}
From Newton's binomial formula
\[	(x+y)^n = \sum_{k=0}^n \binom{n}{k} x^k y^{n-k},	\]
we find after differentiating $m$ times with respect to $x$
\[	(n-m+1)_m (x+y)^{n-m} = \sum_{k=m}^n \binom{n}{k} (k-m+1)_m x^{k-m} y^{n-k}.	\]
If we take $x=1$ and $y=-1$, then for $0 \leq m \leq n-1$
\[	\sum_{k=m}^n \binom{n}{k} (-1)^{n-k} (k-m+1)_m = 0.	\]
Since $(k-m+1)_m$ is a monic polynomial of degree $m$ in $k$, this is equivalent with (\ref{eq:lemma1}). Relation (\ref{eq:lemma2}) can be proven by induction on $n$. The case $n=0$ is trivial. If it holds for $n\geq0$, then the left hand sides of (\ref{eq:lemma2}) for $n+1\geq1$ equals
\begin{eqnarray*}
\sum_{k=0}^{n+1} \binom{n+1}{k} (-1)^{n+1-k} k^{n+1} &=& \sum_{k=1}^{n+1} \binom{n+1}{k} (-1)^{n+1-k} k^{n+1}\\ 
	&=& (n+1)\sum_{k=0}^{n} \binom{n}{k} (-1)^{n-k} (k+1)^{n}.
\end{eqnarray*}
Then write $(k+1)^n = k^n + nk^{n-1}+\cdots$  and by (\ref{eq:lemma1}) we have
\[\sum_{k=0}^{n+1} \binom{n+1}{k} (-1)^{n+1-k} k^{n+1} = (n+1)\sum_{k=0}^{n} \binom{n}{k} (-1)^{n-k} k^n, \]
where the sum is equal to $n!$ by the induction hypothesis. 
\end{proof}

These polynomials $p_n(x;\beta)$ are generalized hypergeometric functions ${}_r F_r$. One can show (by combining the same powers of $x \bmod r$) that
\[ p_n(x;\beta)=\sum_{k=0}^{r-1}\binom{n}{k}\frac{1}{\Gamma(\frac{\beta+k}{r}+1)}(-1)^{n-k}x^k {}_{r} F_r\left(\begin{matrix}
\left(-\frac{n-k-j}{r} \right)_{j=0}^{r-1} &  \\
\left(\frac{k+1+j}{r} \right)_{\substack{j=0 \\ k+j\neq r-1}}^{r-1}, & \frac{\beta+k}{r}+1 
\end{matrix};x^r\right).\]

Due to the orthogonality and normalization of the polynomials $p_n(x;\beta)$, the type I Laguerre-Angelesco polynomials $A_{\vec{n},j}^{(\beta)}$ on and near the diagonal can be expressed in terms of these polynomials.
\begin{theorem}	\label{theorem:diagonal}
The type I Laguerre-Angelesco polynomials on the diagonal $\vec{n}=(n+1,\ldots,n+1)$ are given by
\begin{equation}
	\label{eq:diagonalLA}
	A_{\vec{n},j}(x;\beta) = \frac{r^n}{n!} p_n(\omega^{-j+1}x;\beta), \qquad 1 \leq j\leq r.
\end{equation}
\end{theorem}
\begin{proof}
The degree conditions are clearly satisfied since $p_n$ is of degree $n$. For the orthogonality and normalization, we want the following integral to vanish for $0 \leq k \leq rn+r-2$ and to be equal to $1$ for $k=rn+r-1$
\[	\sum_{j=1}^r \int_0^{\omega^{j-1}\infty} x^k A_{\vec{n},j}(x;\beta) |x|^{\beta}e^{-x^r}\,dx.\]
By (\ref{eq:diagonalLA}) this integral becomes
\[	\frac{r^n}{n!} \sum_{j=1}^r (\omega^{k+1})^{j-1} \int_0^{\infty} x^k p_n(x;\beta) x^{\beta}e^{-x^r}\,dx.\]
Since $\omega$ is the $r^{\mbox{\footnotesize th}}$ root of unity, the sum in this expression is non-zero for only one value of $k$ modulo $r$
\[	\sum_{j=1}^r (\omega^{k+1})^{j-1} = \left\{ \begin{matrix}
	0, & k+1 \not\equiv 0 \mod r, \\
	r, & k+1 \equiv 0 \mod r.
	\end{matrix} \right.\]
Hence, for the orthogonality one only needs to check that
\[ \int_0^{\infty} x^{rj-1} p_n(x;\beta) x^{\beta}e^{-x^r}\,dx = 0,\]
holds for all $1\leq j \leq n$, which is true by (\ref{eq:orthopn}). For the normalization, we need to check the following
\[\frac{r^{n+1}}{n!}\int_0^{\infty} x^{rn+n-1} p_n(x;\beta) x^{\beta}e^{-x^r}\, dx = 1,\]
which is true by (\ref{eq:normpn}). 
\end{proof}

Next, we show that the type I Laguerre-Angelesco polynomials above the diagonal can be written as a linear combination of the polynomials $p_n(x;\beta-j)$ with $0\leq j \leq r-1$.
\begin{theorem}	\label{theorem:above}
Let $\vec{n}=(n,\ldots,n)$ and $\vec{e}_k$ be the $k^{\mbox{\footnotesize th}}$ unit vector in $\mathbb{N}^r$. The type I Laguerre-Angelesco polynomials 
$A_{\vec{n}+\vec{e}_k,j}$ are given by
\[	A_{\vec{n}+\vec{e}_k,j}(x;\beta) = \omega^{-k+1} A_{j-k \mod r} (\omega^{-j+1}x), \qquad 1 \leq j,k \leq r,\]
where the polynomials $A_{\ell}$, $0\leq \ell \leq r-1$ are given by
\begin{equation}	\label{eq:diagonal2}
	\tau_{n,r}(\beta)A_{\ell}(x) = \sum_{j=0}^{r-1}\frac{\omega^{\ell j}}{\nu_n^{(\beta-j)}}p_n(x;\beta-j),
\end{equation}
with normalizing constant
\begin{equation}	\label{eq:tau}
	\tau_{n,r}(\beta) = \frac{n!}{r^n }\Gamma\left(\frac{\beta+n+1}{r}\right),
\end{equation}
and $\nu_n^{(\beta)}$ the leading coefficient of $p_n(x;\beta)$
\begin{equation}	\label{eq:nu}
	\nu_n^{(\beta)} = \frac{1}{\Gamma(\frac{\beta+n}{r}+1)}.
\end{equation}
\end{theorem}
\begin{proof}
We will first determine the degree of the polynomials $A_{\ell}$. For $\ell=0$, we see that $\deg A_{\ell} = \deg p_n(x;\beta)=n$, which implies that $\deg A_{\vec{n}+\vec{e}_j,j}=n$ for all $1\leq j \leq r$. For $\ell=1,2,\ldots,r-1$ the coefficient of $x^n$ on the right hand side of 
(\ref{eq:diagonal2}) is given by
\[	\sum_{j=0}^{r-1} \omega^{\ell j} = \frac{1-\omega^{r\ell}}{1-\omega^{\ell}}=0.\]
Therefore, for all $k \neq j$ we have $\deg A_{\vec{n}+\vec{e}_k,j}<n$ and one can check that it is in fact $n-1$. For the orthogonality and normalization we need the following integral to vanish for all $0\leq \ell \leq rn-1$ and to be equal to $1$ for $\ell=rn$
\[	\sum_{j=1}^r \int_0^{\omega^{j-1}\infty} x^{\ell} A_{\vec{n}+\vec{e}_k,j}(x;\beta) |x|^{\beta}e^{-x^r}\,dx,\]
and this expression is equal to
\[\frac{\omega^{-k+1}}{\tau_{n,r}(\beta)}\sum_{m=0}^{r-1}\frac{\omega^{m(-k+1)}}{\nu_n^{(\beta-m)}}\sum_{j=1}^r 
   \left( \omega^{\ell+m+1}\right)^{j-1} \int_0^{\infty} p_n(x;\beta-m) x^{\ell} x^{\beta} e^{-x^r} \,dx.\]		
The second sum in this expression is
\[	\sum_{j=1}^r \left( \omega^{\ell+m+1}\right)^{j-1} = \left\{ \begin{matrix}
	0, &  \ell+m+1 \not\equiv 0 \mod r, \\
	r, & \ell+m+1 \equiv 0 \mod r.
	\end{matrix}  \right.\]
Therefore, we need to show that
\[	\int_0^{\infty} p_n(x;\beta-m) x^{rj-m-1} x^{\beta} e^{-x^r}\, dx = 0,  \qquad 1 \leq j \leq n,\]
which follows from (\ref{eq:orthopn}). For the normalization we need to show that
\begin{eqnarray*}
		1 &=& \sum_{j=1}^r \int_0^{\omega^{j-1}\infty} x^{rn} A_{\vec{n}+\vec{e}_k,j}(x;\beta) |x|^{\beta}e^{-x^r}\, dx\\
		&=& \frac{1}{\tau_{n,r}(\beta)}\frac{r}{\nu_n^{(\beta-r+1)}} \int_0^{\infty} p_n(x;\beta-r+1) x^{rn+r-1} x^{\beta-r+1} e^{-x^r}\, dx,
\end{eqnarray*}
and this follows from the explicit expression (\ref{eq:tau}) and (\ref{eq:nu}) for $\tau_{n,r}(\beta)$ and $\nu_n^{(\beta)}$ and the expression (\ref{eq:orthopn}) for the integral. 
\end{proof}

We also give an explicit expression for the type I Laguerre-Angelesco polynomials below the diagonal, i.e., for $A_{\vec{n}-\vec{e}_k,j}$ 
with $\vec{n}=(n,\ldots,n)$.

\begin{theorem}	\label{theorem:below}
For every $r>1$ and multi-index $\vec{n}=(n,\ldots,n)$ with $n>0$ we have
\[\gamma_{n,r}(\beta)A_{\vec{n}-\vec{e}_k,j}(x) = \omega^{j-1} \nu_{n-1}^{(\beta)} p_{n-1}(\omega^{-j+1}x;\beta-1)
	-\omega^{k-1} \nu_{n-1}^{(\beta-1)} p_{n-1}(\omega^{-j+1}x;\beta), \]
where $1\leq j,k \leq r $ and the normalizing constant $\gamma_{n,r}(\beta)$ is given by
\begin{equation}	\label{eq:coeffbelow}
	\gamma_{n,r}(\beta) = \frac{(n-1)!}{r^{n-1}\Gamma(\frac{\beta+n-1}{r}+1)}.
\end{equation}
\end{theorem}
\begin{proof}
Note that for $j\neq k$ the degree of $A_{\vec{n}-\vec{e}_k,j}$ is $n-1$, but for $j=k$ the leading term $x^{n-1}$ vanishes and the degree is $n-2$. For the orthogonality conditions and normalization, we need to verify that the following integral vanishes for $0\leq \ell \leq rn-3$ and equals $1$ 
for $\ell = rn-2$
\[	\sum_{j=1}^r \int_0^{\omega^{j-1}\infty} x^{\ell} A_{\vec{n}-\vec{e}_k,j}(x;\beta) |x|^{\beta}e^{-x^r}\, dx,\]
and this expression is equal to
\begin{multline*}
	\frac{1}{\gamma_{n,r}(\beta)}\int_0^{\infty} x^{\ell+\beta}e^{-x^r}\left( \nu_{n-1}^{(\beta)} \sum_{j=1}^r (\omega^{\ell+2})^{j-1} 
            p_{n-1}(x;\beta-1) \right.\\
	\left. - \omega^{k-1} \nu_{n-1}^{(\beta-1)} \sum_{j=1}^r (\omega^{\ell+1})^{j-1}p_{n-1}(x;\beta)  \right)\, dx.
\end{multline*}
The two sums involving the roots of unity $\omega$ are 
\[\sum_{j=1}^r (\omega^{\ell+2})^{j-1}= \left\{ \begin{matrix}
	0, & \ell +2 \not\equiv 0 \mod r, \\
	r, & \ell +2  \equiv 0 \mod r,
	\end{matrix} \right.\]
and 
\[	\sum_{j=1}^r (\omega^{\ell+1})^{j-1}= \left\{ \begin{matrix}
	0, & \ell +1 \not\equiv 0 \mod r, \\
	r, & \ell +1  \equiv 0 \mod r,
\end{matrix} \right.   \]
so the integral vanishes when $\ell +2 \not\equiv 0 \mod r$ and $\ell+1 \not\equiv 0 \mod r$. 
In case $\ell =rj-2$ for $1\leq j \leq n-1$ we need to verify
\[	\int_0^{\infty} p_{n-1}(x;\beta-1) x^{rj-2+\beta}e^{-x^r}\, dx = 0,  \]
which holds by (\ref{eq:orthopn}). In case $\ell=rj-1$ for $1\leq j \leq n-1$ we need 
\[	\int_0^{\infty} p_{n-1}(x;\beta)x^{rj-1+\beta}e^{-x^r}\,  dx=0,\]
and this again follows from (\ref{eq:orthopn}). For the normalization we need to show that 
\begin{eqnarray*}
		1 &=& \sum_{j=1}^r \int_0^{\omega^{j-1}\infty} x^{rn-2} A_{\vec{n}-\vec{e}_k,j}(x;\beta) |x|^{\beta}e^{-x^r}\,dx \\
		&=& \frac{r}{\gamma_{n,r}(\beta)}\nu_{n-1}^{(\beta)} \int_0^{\infty} p_{n-1}(x;\beta-1) x^{rn-2+\beta}e^{-x^r}\,dx,
\end{eqnarray*}
and this follows from the normalization of $p_n(x;\beta)$ in (\ref{eq:normpn}) and the explicit expressions of $\gamma_{n,r}(\beta)$ 
(\ref{eq:coeffbelow}) and $\nu_{n-1}^{(\beta)}$ (\ref{eq:nu}). 
\end{proof}
Similar as in the case $r=1$, see equation (\ref{JA-LA1}), one can show that the type I Laguerre-Angelesco polynomials $A_{\vec{n},j}(x;\beta)$ on the $r$-star (for any multi-index $\vec{n}$) are limiting cases of the type I Jacobi-Angelesco polynomials $B_{\vec{n},j}(x;\alpha,\beta)$ as follows
\[ 	A_{\vec{n},j}(x;\beta) = \lim_{\alpha\rightarrow \infty} \alpha^{-\frac{|\vec{n}|+\beta}{r}} B_{\vec{n},j}(\alpha^{-\frac{1}{r}}x;\alpha,\beta). \]
One can verify this by checking the orthogonality conditions and the normalization, or by taking the limit in the expression \eqref{pnJA} 
which gives \eqref{eq:pn}.

\subsubsection{Recurrence relation}
For any $r$, the nearest neighbor recurrence relations \cite{VA} for the type I Laguerre-Angelesco polynomials are given by
\begin{equation}	 \label{eq:nnrrI}
x A_{\vec{n},j}(x;\beta) = A_{\vec{n}-\vec{e}_k,j}(x;\beta) + b_{\vec{n}-\vec{e}_k,k}A_{\vec{n},j}(x;\beta) + \sum_{\ell=1}^r a_{\vec{n},\ell} A_{\vec{n}+\vec{e}_{\ell},j}(x;\beta), 
\end{equation}
for every $1\leq j,k\leq r$. There is a similar recurrence relation for the type II Laguerre-Angelesco polynomials, but with a shift in the coefficients $b_{\vec{n},k}$ (see Section \ref{section:rectypeII}). An explicit expression for the recurrence coefficients for the type I Laguerre-Angelesco polynomials near the diagonal is given by
\begin{proposition}	\label{prop:reccoef}
Let $\vec{n}=(n\ldots,n)$ be a diagonal multi-index for $r\geq 1$. Then 
\[	a_{\vec{n},k} = \frac{n}{r^2} \frac{\Gamma(\frac{n+\beta+1}{r})}{\Gamma(\frac{n+\beta-1}{r}+1)}\omega^{2(k-1)}.\]
Moreover, for $r>1$
\[	b_{\vec{n}-\vec{e}_k,k} = \frac{\Gamma(\frac{n+\beta-1}{r}+1)}{\Gamma(\frac{n+\beta-2}{r}+1)}\omega^{k-1}. \]
\end{proposition}

\begin{proof}
If we write for any $\vec{n}=(n_1,\ldots,n_r)$
\[	A_{\vec{n},k}(x;\beta) = \kappa_{\vec{n},k} x^{n_k-1} + \delta_{\vec{n},k} x^{n_k-2}+\cdots ,\]
then for $\vec{n}=(n,\ldots,n)$ we get, by comparing the coefficient of $x^n$ in (\ref{eq:nnrrI}),
\[	a_{\vec{n},k} = \frac{\kappa_{\vec{n},k}}{\kappa_{\vec{n}+\vec{e}_k,k}},\]
and this ratio can be evaluated by using Theorem \ref{theorem:diagonal} which gives 
\[	\kappa_{\vec{n},k} = \frac{r^{n-1}}{(n-1)!}\frac{1}{\Gamma(\frac{n+\beta-1}{r}+1)} (\omega^{-k+1})^{n-1},\]
and Theorem \ref{theorem:above} gives
\[	\kappa_{\vec{n}+\vec{e}_k,k} = \frac{r^{n+1}}{n! \Gamma(\frac{n+\beta+1}{r})} (\omega^{-k+1})^{n+1}.\]
The coefficients $b_{\vec{n}-\vec{e}_k,k}$ can also be computed by comparing coefficients, but the computations are a bit longer and only the case $r>1$ is covered. For $r=1$, the computations are slightly different, but in this case the result is known because this corresponds to the Laguerre polynomials on $[0,\infty)$ (to be more precise, this corresponds to rescaled Laguerre polynomials). By comparing the coefficient of $x^{n-1}$ in (\ref{eq:nnrrI}), we obtain the expression for $b_{\vec{n}-\vec{e}_k,k}$
\[	b_{\vec{n}-\vec{e}_k,k} = \frac{\delta_{\vec{n},k}}{\kappa_{\vec{n},k}} - \frac{\delta_{\vec{n}+\vec{e}_k,k}}{\kappa_{\vec{n}+\vec{e}_k,k}} 
    - \sum_{l=1,l\neq k}^{r}\frac{\kappa_{\vec{n},l}}{\kappa_{\vec{n},k}}\frac{\kappa_{\vec{n}+\vec{e}_l,k}}{\kappa_{\vec{n}+\vec{e}_k,k}}.\]
From Theorem \ref{theorem:diagonal}, we have
\[	\delta_{\vec{n},k} = - \frac{r^{n-1}}{(n-2)!\Gamma(\frac{n+\beta-2}{r}+1)}\omega^{(-k+1)(n-2)},  \qquad n\geq 2,\]
and from Theorem \ref{theorem:above}
\[	\delta_{\vec{n}+\vec{e}_k,k} = -\frac{r^n}{(n-1)!\Gamma(\frac{n+\beta+1}{r})} \sum_{j=0}^{r-1} \frac{\Gamma(\frac{n+\beta-j}{r}+1)}
      {\Gamma(\frac{n+\beta-j-1}{r}+1)}\omega^{n(-k+1)},\]
and for $\ell \neq k$
\[	\kappa_{\vec{n}+\vec{e}_{\ell},k} = - \frac{r^n}{(n-1)!\Gamma(\frac{n+\beta+1}{r})}  \sum_{j=0}^{r-1} \frac{\Gamma(\frac{n+\beta-j}{r}+1)}{\Gamma(\frac{n+\beta-j-1}{r}+1)} \omega^{-{\ell}+1+(n-1)(-k+1)+j(k-{\ell})}.\]
Inserting this in the expression of $b_{\vec{n}-\vec{e}_k,k}$ gives the desired expression. 
\end{proof}

\subsubsection{Differential equation}
In this section we will give a linear differential equation of order $r+1$ for the polynomial $p_n(x;\beta)$ given in (\ref{eq:pn}). The differential equation is a combination of lowering and raising operators for these polynomials, i.e., differential operators that lower or raise the degree of the polynomial and raise/lower the parameter $\beta$.
\begin{lemma}
For the polynomials $p_n(x;\beta)$ given in (\ref{eq:pn}), one has for $\beta >-1$ the lowering operator
\begin{equation}	\label{eq:lower}
	p_n'(x;\beta) = np_{n-1}(x;\beta+1),
\end{equation}
and the raising operator
\begin{equation} 	\label{eq:raising}
	(x^{\beta}e^{-x^r}p_n(x;\beta))' = x^{\beta-1}e^{-x^r} \sum_{k=1}^r (-1)^k \binom{r}{k}r x^{r-k} p_{n+k}(x;\beta-k).
\end{equation}
\end{lemma}
\begin{proof}
The lowering operator (\ref{eq:lower}) can easily be proven by differentiating (\ref{eq:pn}). To prove (\ref{eq:raising}) we first observe that
\[	\left[x^{\beta}e^{-x^r}p_n(x;\beta) \right]' = x^{\beta-1} e^{-x^r} \pi_{n+r}(x),\]
where $\pi_{n+r}$ is a polynomial of degree $n+r$ given by
\begin{equation}	\label{eq:2}
\pi_{n+r}(x) = (\beta-rx^r) p_n(x;\beta) + xp_n'(x;\beta).
\end{equation} 
Integrating both sides shows that for $\beta>0$
\begin{eqnarray*}
	\int_0^{\infty} x^{\beta-r} e^{-x^r} \pi_{n+r}(x) x^{r(j+1)-1} dx &=& \int_0^{\infty} 
                       \left[x^{\beta}e^{-x^r}p_n(x;\beta) \right]' x^{rj} dx \\
		&=& -rj \int_0^{\infty}x^{\beta}e^{-x^r}p_n(x;\beta) x^{rj-1} dx,
\end{eqnarray*}
and by (\ref{eq:orthopn}) this is zero for $1\leq j \leq n$ but also for $j=0$. So $\pi_{n+r}$ is a polynomial of degree $n+r$ which is orthogonal to all $x^{rj-1}$ for $1\leq j\leq n+1$ with weight $x^{\beta-r} e^{-x^r}$ on the interval $[0,\infty)$. These are $n+1$ orthogonality conditions. The $r$ polynomials $x^{r-k}p_{n+k}(x;\beta-k)$, $1\leq k \leq r$, have the same orthogonality conditions, they are all of degree $n+r$ and they are linearly independent. Hence, they span the linear space of polynomials of degree $n+r$ with the $n+1$ orthogonality conditions. Therefore
\begin{equation}	\label{eq:1}
	\pi_{n+r}(x) = \sum_{k=1}^r a_k x^{r-k} p_{n+k}(x;\beta-k),
\end{equation}
for some $a_k$, $k=1,\ldots,r$. To find these coefficients, one compares the coefficients of $x^{r-k}$ in (\ref{eq:1}) and (\ref{eq:2}). 	
\end{proof}
With these operators one can find the differential equation.

\begin{theorem}		\label{theorem:diffpn}
For any $n\in \mathbb{N}$ and $r\geq1$, the polynomial $y=p_n(x;\beta)$ satisfies the differential equation
\[	xy^{(r+1)} + (\beta+r)y^{(r)} + \sum_{k=0}^r c_{k,n} x^k y^{(k)} = 0,	\]
where
\begin{equation}	\label{eq:diffcoeff}
	c_{k,n} = (-1)^{r+1+k}\binom{r}{k}r (n-r+1)_{r-k}.
\end{equation}
\end{theorem}
\begin{proof}
From the lowering operator (\ref{eq:lower}) one has
\[	p_n^{(r)}(x;\beta) = \frac{n!}{(n-r)!}p_{n-r}(x;\beta+r).\]
Multiplying both sides by $x^{\beta+r}e^{-x^r}$ and differentiating gives
\begin{multline*}
x^{\beta+r}e^{-x^r} p_{n}^{(r+1)}(x;\beta) + x^{\beta+r-1} e^{-x^r} [\beta+r-rx^r] p_n^{(r)}(x;\beta) \\
= \frac{n!}{(n-r)!} x^{\beta+r-1} e^{-x^r} \sum_{k=1}^r (-1)^k \binom{r}{k}r x^{r-k} p_{n-r+k}(x;\beta+r-k),
\end{multline*}
where we used the raising operator (\ref{eq:raising}). Using the lowering operator (\ref{eq:lower}) one has
\[	p_{n-r+k}(x;\beta+r-k) = \frac{(n-r+k)!}{n!} p_n^{(r-k)}(x;\beta),\]
hence we have
\[	x p_{n}^{(r+1)}(x;\beta) +  (\beta+r-rx^r) p_n^{(r)}(x;\beta) 
-\sum_{k=1}^r (-1)^k \binom{r}{k}r \frac{(n-r+k)!}{(n-r)!} x^{r-k}p_n^{(r-k)}(x;\beta)=0,\]
or 
\[	xp_n^{(r+1)}(x;\beta) + (\beta+r)p_n^{(r)}(x;\beta) + \sum_{k=0}^rc_{k,n} x^k p_n^{(k)}(x;\beta) = 0.\]
\end{proof}

\subsubsection{Asymptotic zero behavior}
Now we investigate the asymptotic behavior of the zeros of the type I Laguerre-Angelesco polynomials $A_{\vec{n},j}(x;\beta)$, $1\leq j \leq r$, for the multi-index $\vec{n}=(n,n,\ldots,n)$ and $n\rightarrow\infty$. From Theorem \ref{theorem:diagonal} it is clear that the zeros of $A_{\vec{n},j}$ are copies of the zeros of $p_n$ (which are on $(0,\infty)$, see Lemma \ref{lemma1}), but rotated to the interval $(0,\omega^{j-1}\infty)$. Hence, we only need to investigate the behavior of the zeros of $p_n$ given in (\ref{eq:pn}). First we prove that the zeros of $p_n$ are all in $(0,\infty)$ whenever $\beta>-1$.
\begin{lemma}	\label{lemma1}
Let $\beta>-1$, then all the zeros of $p_n(x;\beta)$ given in (\ref{eq:pn}) are simple and lie in the open interval $(0,\infty)$.
\end{lemma}
\begin{proof}
The proof is a modification of the zero location for usual orthogonal polynomials, see e.g., \cite[Thm. 3.3.1]{Szego}.
Suppose $x_{1},\ldots,x_m$ are the zeros of odd multiplicity of $p_n(x;\beta)$ that lie in $(0,\infty)$ and that $m<n$. Then consider the polynomial $q_m(x) = (x^r-x_1^r)\cdots(x^r-x_m^r)$. This is a polynomial of degree $rm$ with $m$ real zeros on $(0,\infty)$ at the points $x_1,\ldots,x_m$ and no other sign changes on $(0,\infty)$. Hence, the product $p_n(x;\beta)q_m(x)$ has a constant sign on $(0,\infty)$ so that 
\[	\int_0^{\infty} p_n(x;\beta)q_m(x) x^{\beta+r-1}e^{-x^r} dx \neq 0.\]
But $q_m(x) x^{r-1}$ contains only powers $x^{rj-1}$ with $1\leq j \leq m+1 \leq n$, hence by (\ref{eq:orthopn}) this integral is zero. This contradiction implies that $m\geq n$ and since $p_n$ is of degree $n$, we see that $m=n$. 
\end{proof}
The above lemma states that the zeros of $p_n(x;\beta)$ lie on the positive half-line. Moreover, they will tend to $\infty$ as $n\rightarrow \infty$. To be able to say something about the asymptotic behavior of the zeros, we will obtain a proper rescaling such that the zeros lie in a compact interval. Note that for $r=1$, we have the classical Laguerre polynomials. In that case, the rescaling is known to be $4n$, since the largest zero of the Laguerre polynomial grows as $4n$. For any $r\geq 1$, we first show that after a rescaling 
\[	x = \alpha_r n^{1/r} z,\]
for some constant $\alpha_r>0$ (which only depends on $r$), the zeros will lie in a compact interval for $n$ large enough. This will be shown by using the infinite-finite range inequality. This inequality says that there is a neighborhood where the $L^1(\mathbb{R}^+)$-norm lives and only a small fraction of the norm lies outside this neighborhood. To make sure this compact interval is always $[0,1]$, we show that the largest zero (as $n$ tends to infinity) grows like $\left(\frac{r+1}{r} \right)^{\frac{r+1}{r}}n^{\frac1r}$, which gives the proper rescaling factor.

\begin{lemma}   \label{lemma:ifri}
For $n$ large enough, the zeros of the polynomials $p_n$ lie in the interval
\[	V_{n,r}:= \left[0,(r+1)^{1/r}\left(\frac{r+1}{r}\right)^{\frac{r+1}{r}}n^{1/r}\right].	\]
\end{lemma}
\begin{proof}
Denote $0<x_{1,n}<\cdots<x_{m,n}$ the zeros of $p_n$ in $V_{n,r}$ and assume that $m<n$ so that there are zeros that lie outside this interval. Define a new polynomial
\[	q_m(x)= (x^r-x_{1,n}^r)\cdots(x^r-x_{m,n}^r),	\]
which is of degree $rm<rn$.	Since the zeros of $p_n$ are simple (Lemma \ref{lemma1}), the product $p_nq_m$ does not change sign on $V_{n,r}$. Moreover, since $q_m$ is of degree $rm <rn$ we have, due to the orthogonality of $p_n$,
\begin{equation}	\label{eq:orthopq}
	0  = \int_0^{\infty} p_n(x) q_m(x) x^{\beta+r-1} e^{-x^r} dx.  
\end{equation}
Now consider the weight function $w(x) = e^{-x^r}$ on $[0,\infty)$, then the support of the equilibrium measure $\mu_w$ corresponding to this weight function is given by
\[	\mathcal{S}_w = \left[0,\frac{2}{r^{1/r}} \binom{r-1/2}{\lfloor\frac{r}{2}\rfloor}^{-1/r} \right],	\] 
see \cite{SaffTotik} for the computations. Then the infinite-finite range inequality \cite[Thm. 6.1 in Ch. III]{SaffTotik},
\cite[Thm. 4.1]{LevinLub} tells us that for a polynomial $P_N$ of degree at most $N$ the following $L^1(\mathbb{R}^+)$-norm lives on any neighborhood 
$\mathcal{N}$ of $\mathcal{S}_w$
\[    \int_{\mathbb{R}^*\setminus \mathcal{N}} |P_N(x)| e^{-Nx^r} x^\beta\, dx \leq De^{-dn} \int_{\mathcal{N}} |P_N(x)| e^{-Nx^r} x^\beta\, dx, \]
where $D$ and $d$ are two positive constants, independent of $N$ and $P_N$. Choose $N$ large enough so that $De^{-dn} < 1$, then
\begin{equation}	\label{eq:ifri}
	\int_{\mathbb{R}^+\setminus \mathcal{N}} |P_N(x)|e^{-Nx^r}x^{\beta}\, dx < \int_{\mathcal{N}} |P_N(x)|e^{-Nx^r}x^{\beta}\, dx.
\end{equation}
To apply the infinite-finite range inequality (\ref{eq:ifri}) on the integral (\ref{eq:orthopq}), set $N=n(r+1)$ and
\[	P_N(y) = p_n(n^{1/r}(r+1)^{1/r}y)q_m(n^{1/r}(r+1)^{1/r}y)y^r.	\]
Then by rescaling $x= n^{1/r}(r+1)^{1/r}y$ the integral (\ref{eq:orthopq}) becomes
\[	0= \int_{V_{r}^*} P_N(y) y^{\beta-1} e^{-Ny^r}\, dy +\int_{\mathbb{R}^+\setminus V^*_r} P_N(y) y^{\beta-1} e^{-Ny^r}\, dy,\]
where $V_{r}^* = \left[0,\left(\frac{r+1}{r}\right)^{\frac{r+1}{r}}\right]$ and one can check that this set is a neighborhood of $\mathcal{S}_w$. Since $p_nq_m$ does not change sign on $V_{n,r}$, $P_N$ has the same sign on $V_{r}^* = n^{-1/r}(r+1)^{-1/r}V_{n,r}$. Suppose that $P_N$ is positive on this set, then by the infinite-finite range inequality (\ref{eq:ifri})
\begin{eqnarray*}
		0&=&\int_{V_{r}^*} |P_N(y)| y^{\beta-1} e^{-Ny^r}\, dy +\int_{\mathbb{R}^+\setminus V^*_{r}} P_N(y) y^{\beta-1} e^{-Ny^r}\, dy \\
		&>&\int_{V_{r}^*} |P_N(y)| y^{\beta-1} e^{-Ny^r}\, dy - \int_{V_{r}^*} |P_N(y)| y^{\beta-1} e^{-Ny^r}\, dy = 0,
\end{eqnarray*}
which gives a contradiction, so $m\geq n$ and since $\deg p_n = n$ we have $m=n$. Similarly for $P_N$ negative on $V_{r}^*$, we obtain a contradiction so that $m=n$. 
\end{proof}
Note that the previous lemma does not give the smallest possible interval that contains the zeros since this set was chosen to suit the neighborhood $V_{r}^*$ and the polynomial $P_N$. Indeed, for $r=1$ we obtain the interval $V_{n,1} = [0,8n]$ but we know that the zeros lie in the interval $[0,4n]$ for $n$ large enough. Since the zeros of $p_n$ lie in an interval $[0,\alpha_r n^{1/r}]$ for a constant $\alpha_r>0$, we only need to determine the smallest possible $\alpha_r$. To obtain this optimal $\alpha_r$, we use the rescaling
\[x= \alpha_r n^{\frac{1}{r}} z,\]
such that the zeros of $p_n$ lie in the compact interval $[0,1]$ and obtain a growing rate for the largest zero of $p_n$. The inverse of this growth rate will determine the optimal $\alpha_r$.

From here on, we take a large enough $N$ such that for all $n\geq N$ all the zeros of the rescaled $p_n$ lie in $[0,1]$. Hence, define the polynomial
\[ 	\tilde{p}_n(x;\beta) := p_n\left( \alpha_r n^{\frac{1}{r}} x;\beta \right). \]
Then we will analyze the behavior of the zeros of these new polynomials. Note that from the differential equation of $p_n(x;\beta)$ in Theorem \ref{theorem:diffpn}, one can derive the following differential equation for $\tilde{p}_n$
\begin{equation}
\label{eq:difftildep}
x \tilde{p}_n^{(r+1)}(x;\beta) + (\beta+r) \tilde{p}_n^{(r)}(x;\beta) + \alpha_r^r n  \sum_{k=0}^r c_{k,n} x^k \tilde{p}_n^{(k)}(x;\beta)=0.  
\end{equation}
Denote the zeros of $\tilde{p}_n(x;\beta)$ by $0< x_{1,n} <  \cdots < x_{n,n} $ and consider the normalized zero counting measure given by
\[	\mu_n = \frac{1}{n} \sum_{j=1}^n \delta_{x_{j,n}}.\]
Then $(\mu_n)_n$ is a sequence of probability measures on $[0,1]$, and by Helley's selection principle it will contain a subsequence $(\mu_{n_k})_k$ that converges weakly to a probability measure $\mu$ on $[0,1]$. If this limit is independent of the subsequence, then we call it the asymptotic zero distribution of the zeros of $\tilde{p}_n(x;\beta)$. We will investigate this by means of the Stieltjes transform
\[	S_n(z) = \int_0^{1} \frac{d\mu_n(x)}{z-x} = \frac{1}{n}\frac{\tilde{p}_n'(z;\beta)}{\tilde{p}_n(z;\beta)}, \qquad S(z) = \int_0^{1} \frac{d\mu(x)}{z-x}, \qquad z \in \mathbb{C}\setminus [0,1],\]
and use the Grommer-Hamburger theorem \cite{Geronimo} which says that $\mu_n$ converges weakly to $\mu$ if and only if $S_n$ converges uniformly on compact subsets of $\mathbb{C}\setminus [0,1]$ to $S$ and $zS \rightarrow 1$ as $z \rightarrow \infty$.

First we show that the weak limit of $(\mu_{n_k})_k$ has a Stieltjes transform $S$ which satisfies an algebraic equation of order $r+1$. 
\begin{proposition}
Suppose that $\mu_{n_k}$ converges weakly to $\mu$, then the Stieltjes transform $S$ of $\mu$ satisfies
\begin{equation}	\label{eq:algS}
	zS^{r+1} - r\alpha_r^r(zS-1)^r=0.
\end{equation}
\end{proposition}
\begin{proof}
Since the Stieltjes transform $S_n$ of $\mu_n$ satisfies
\[	\tilde{p}_n'(z;\beta) = n \tilde{p}_n(z;\beta) S_n(z),\]
one can prove by induction that the higher order derivatives of $\tilde{p}_n(z;\beta)$ can be described in terms of $S_n$ and its derivatives as follows
\begin{equation}	\label{diffpn}
\tilde{p}_n^{(k)}(z;\beta) = n^k \tilde{p}_n(x;\beta) \left[S_n^k(z) + \frac{1}{n} G_{n,k}(S_n,S_n',\ldots,S_n^{(k-1)})\right], \qquad k\geq 0,
\end{equation}
where $G_{n,k}$ is a polynomial in $k$ variables with coefficients of order $\mathcal{O}(1)$ in $n$. This polynomial is given recursively by
\begin{multline}    \label{eq:Gfunction}
	G_{n,k}(x_1,x_2,\ldots,x_k) = x_1 G_{n,k-1}(x_1,x_2,\ldots,x_{k-1}) + (k-1)x_1^{k-2}x_2 \\
	+ \frac{1}{n} \sum_{j=1}^{k-1} \frac{\partial}{\partial x_j} G_{n,k-1} (x_1,x_2,\ldots,x_{k-1})x_{j+1},
\end{multline}
with $G_{n,0}=0$. 
	
Inserting (\ref{diffpn}) in the differential equation of $\tilde{p}_n$ (\ref{eq:difftildep}) gives
\begin{multline*}
	0= z n^{r+1} \tilde{p}_n \left[S_n^{r+1} + \frac{1}{n} G_{n,r+1}(S_n,S_n',\ldots,S_n^{(r)})\right] \\+(\beta+r)n^r \tilde{p}_n 
    \left[S_n^{r} + \frac{1}{n} G_{n,r}(S_n,S_n',\ldots,S_n^{(r-1)})\right]\\
	+  \alpha_r^r  \sum_{k=0}^r c_{k,n} z^k n^{k+1} \tilde{p}_n \left[S_n^k + \frac{1}{n} G_{n,k}(S_n,S_n',\ldots,S_n^{(k-1)})\right],
\end{multline*}
Next, divide by $ n^{r+1}\tilde{p}_n$ and let $n=n_k \rightarrow \infty$. Since $S_{n_k}$ converges uniformly to $S$ on compact subsets of $\mathbb{C}\setminus [0,1]$, it follows that all the derivatives $S_{n_k}^{(j)}$ converge uniformly to $S^{(j)}$ on compact subsets of $\mathbb{C}\setminus [0,1]$. Furthermore, from (\ref{eq:diffcoeff})
\[	\lim_{n\rightarrow\infty} \frac{c_{k,n}}{n^{r-k}} = (-1)^{r+k+1}r \binom{r}{k},\]
so that in the limit we find
\[	zS^{r+1} + r\alpha_r^r  \sum_{k=0}^r  (-1)^{r+k+1} \binom{r}{k} z^kS^k = 0.\]
By using the binomial theorem, one can find equation (\ref{eq:algS}). Note that the equation does not depend on the subsequence $(n_k)_k$ anymore, so every converging subsequence has a limit $S$ satisfying (\ref{eq:algS}). 
\end{proof}
The algebraic equation from the previous proposition is of order $r+1$ so it has $r+1$ solutions, but we are interested in the solution which is a Stieltjes transform of a probability measure on $[0,1]$. Then one can find the asymptotic zero distribution measure $\mu$ by using the Stieltjes-Perron inversion formula.
\begin{theorem}   \label{theorem:asympI}
For $\alpha_r = \left(\frac{r+1}{r}\right)^{\frac{r+1}{r}}$, the zeros of $\tilde{p}_n(x;\beta)$ lie in $[0,1]$ for $n$ large enough and the asymptotic zero distribution of $\tilde{p}_n(x;\beta)$ as $n\rightarrow \infty$ is given by a measure which is absolutely continuous on $[0,1]$ with a density $u_r$ given by $u_r(x) = x^{r-1} w_r(x^r)$, where $w_r$ is given by
\[	w_r(\hat{x}) = \frac{1}{\pi \hat{x}} \frac{\sin \theta \sin r\theta \sin(r+1)\theta}{|\sin(r+1)\theta e^{i\theta} -\sin r\theta |^2},\]
where we used the change of variables
\[	\hat{x} = x^r = \frac{1}{c_r}\frac{(\sin (r+1)\theta)^{r+1}}{ (\sin r\theta)^r \sin \theta}, \qquad 0 < \theta < \frac{\pi }{r+1},\]	
and $c_r =\frac{(r+1)^{r+1}}{r^r}$.
\end{theorem}

\begin{proof}
The proof is along the same lines as in \cite{LeursVA} where the asymptotic zero distribution was obtained for Jacobi-Angelesco polynomials and as in \cite{NeuschelVA} for Jacobi-Pi\~neiro polynomials. First we transform the algebraic equation (\ref{eq:algS}) by taking
\[	W = \frac{zS}{zS-1}, \qquad zS = \frac{W}{W-1},\]
which gives
\begin{equation}	\label{eq:algW}
	W^{r+1} - r\alpha_r^r z^r W + r\alpha_r^r z^r=0.
\end{equation} 
We look for a solution $W$ of the form $\rho e^{i \theta}$, where $\theta$ is real and $\rho>0$. Take $z=x \in [0,1]$ and insert $W = \rho e^{i\theta}$ into (\ref{eq:algW}) to find
\[	\rho^{r+1} e^{i(r+1)\theta} - r\alpha_r^r x^r\rho e^{i\theta} + r\alpha_r^r x^r=0.\]
Hence, the real and imaginary parts satisfy
\begin{eqnarray}
	\label{eq:real}
	\rho^{r+1} \cos(r+1)\theta - r\alpha_r^rx^r\rho \cos\theta + r\alpha_r^r x^r&=&0, \\
	\label{eq:im}
	\rho^{r+1} \sin(r+1)\theta - r\alpha_r^r x^r\rho \sin\theta &=&0.
\end{eqnarray}
From (\ref{eq:im}) we have
\[	\hat{x} := x^r = \frac{\rho^r \sin(r+1)\theta}{r\alpha_r^r \sin\theta},\]
and using this in (\ref{eq:real}), we find
\[	\rho(\hat{x}) = \frac{\sin(r+1)\theta}{\sin r\theta}. \]	
Note that when $\theta\in (0,\pi/(r+1))$ one has $\rho >0$. Combining this gives
\[	\hat{x} = x^r = \frac{1}{r\alpha_r^r}\frac{(\sin (r+1)\theta)^{r+1}}{ (\sin r\theta)^r \sin \theta}.\]
This function gives the zeros in terms of the parameter $\theta$. By differentiation, one can see that it has its maximum in $\theta =0$. By applying de l`H\^{o}pital's rule $r+1$ times, one obtains $\hat{x}(0) = \frac{(r+1)^{r+1}}{r^{r+1}\alpha_r^r}$. Therefore, if one takes
\[   \alpha_r =\left(\frac{r+1}{r}\right)^{\frac{r+1}{r}}, \]
the zeros of $\tilde{p}_{n}$ will lie in $[0,1]$ for $n$ large enough. So for $\hat{x} \in [0,1]$, equation (\ref{eq:algW}) has a solution of the 
form $\rho e^{i\theta}$. Observe that also $\rho e^{-i\theta}$ is a solution and in fact
\[	W_+(\hat{x}) = \lim_{\epsilon \rightarrow 0+} W(\hat{x}+i\epsilon) = \rho(\hat{x}) e^{i\theta} ,
  \qquad 	W_-(\hat{x}) = \lim_{\epsilon \rightarrow 0+} W(\hat{x}-i\epsilon) = \rho(\hat{x}) e^{-i\theta},\]
are the boundary values of the function $W$ which is analytic on $\mathbb{C} \setminus [0,1]$. From the Stieltjes-Perron inversion formula, we can compute the density $u_r$ as
\[	u_r(x) = \frac{1}{2\pi i} (S_-(x) - S_+(x)) = \frac{\rho}{\pi x} \frac{\sin\theta}{|\rho e^{i\theta}-1|^2}.\]
Writing everything in terms of $\hat{x}$ gives the required result. 
\end{proof}

\begin{figure}[t]
\centering 
\includegraphics[scale=0.5]{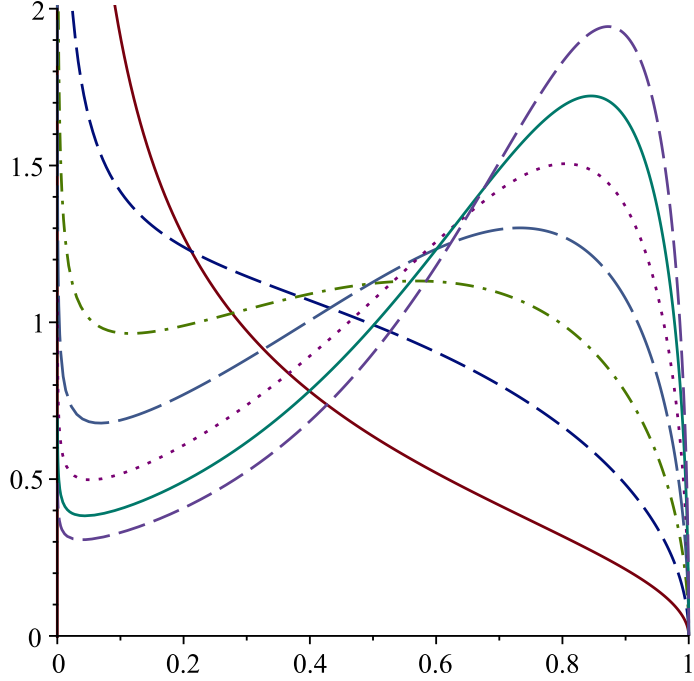}
\caption{The density $u_r$ of the asymptotic zero distribution on $[0,1]$ for $r=1$ (solid in red), $r=2$ (dash in blue), $r=3$ (dash-dot), 
$r=4$ (long dash), $r=5$ (dots), $r=6$ (solid in green) and $r=7$ (dash in purple).}
\label{figure:AZD}
\end{figure}

In Fig. \ref{figure:AZD} we have plotted the densities $u_r$ on $[0,1]$ for $r=1,2,3,4,5,6,7$. The case $r=1$ corresponds to the well-known Marchenko-Pastur distribution for the Laguerre polynomials with density
\[
u_1(x) = \frac{2}{\pi} x^{-1/2}(1-x)^{1/2}, \qquad 0 < x < 1.
\]
For $r\geq 1$, the density $u_r$ behaves near the endpoints as follows
\begin{eqnarray*}
u_r(x) &\sim&	x^{-\frac1{r+1}}, \qquad x \rightarrow 0,\\
u_r(x) &\sim& (1-x)^{1/2}, \qquad x \rightarrow 1.
\end{eqnarray*}
So for a fixed $r$ the density $u_r$ only has a singularity in $x=0$ of order $\frac{1}{r+1}$. This shows the typical behavior of zeros for an Angelesco system were the zeros on all the other intervals $[0,\omega^{j-1}]$, $j=2,3, \ldots,r$ push the zeros on $[0,1]$ to the right.

\subsection{Type II Laguerre-Angelesco polynomials}	\label{secIIsubII}
The type II Laguerre-Angelesco polynomial $L_{\vec{n}}(x;\beta)$ for the multi-index\\ $(n_1,n_2,\ldots,n_r)$ on the $r$-star and parameter $\beta>-1$ is the monic polynomial of degree $|\vec{n}|$ defined by the orthogonality conditions
\begin{equation}
\label{eq:orthotypeIILA}
\int_0^{\omega^{j-1}\infty} x^k L_{\vec{n}}(x;\beta) |x|^{\beta} e^{-x^r}\, dx =0, \qquad 0 \leq k \leq n_j-1,
\end{equation}
for all  $1 \leq j \leq r$. These polynomials were already studied by Sorokin \cite{Sorokin4,Sorokin} for simultaneous Pad\'{e} approximants, but he only considered the polynomials with a multi-index $\vec{n}=(n,n,\ldots,n)$. For the polynomials on and above the diagonal, we state the Rodrigues formula and compute the explicit expressions. With these expressions we compute the recurrence coefficients of the nearest neighbor recurrence relations. Then we state a differential equation of order $r+1$ for the diagonal polynomials which will be used to determine the asymptotic behavior of the zeros of these polynomials.

\subsubsection{Explicit expression}
The explicit expression for the type II Laguerre-Angelesco polynomials on the diagonal can be derived from its Rodrigues formula, which was stated in \cite{Sorokin4}.
\begin{theorem}
The type II Laguerre-Angelesco polynomials on the diagonal, $\vec{n}=(n,n,\ldots,n)$, is given by the Rodrigues formula
\begin{equation}	\label{eq:RodriguesLA}
	x^{\beta} e^{-x^r} L_{\vec{n}}(x;\beta) = \frac{(-1)^n}{r^n} \frac{d^n}{dx^n} [x^{\beta+n} e^{-x^r}].
\end{equation}
\end{theorem}
\begin{proof}
One can check easily that the degree of $L_{\vec{n}}(x;\beta)$ in (\ref{eq:RodriguesLA}) is $rn$ and that this polynomial is monic. For the orthogonality take $1 \leq j \leq r$ and $0 \leq k \leq n-1$, then we have
\[     \int_0^{\omega^{j-1}\infty} x^k L_{\vec{n}}(x;\beta) |x|^{\beta} e^{-x^r}\, dx 
   = \frac{(-1)^n}{r^n}\int_0^{\omega^{j-1}\infty} x^k \frac{d^n}{dx^n} x^{\beta+n} e^{-x^r}\, dx.\]
Using integration by parts $k$ times gives
\[	\frac{(-1)^{n-k}}{r^n} k!\left( \frac{d^{n-k}}{dx^{n-k}} x^{\beta+n} e^{-x^r} \right) \bigg\rvert_{0}^{\omega^{j-1}\infty}=0.\]
\end{proof}
From the Rodrigues formula one can compute the explicit expression for the type II Laguerre-Angelesco polynomials on the diagonal 
$\vec{n}=(n,n,\ldots,n)$
\begin{equation}	\label{eq:typeII}
L_{\vec{n}}(x;\beta) = \frac{(-1)^n}{r^n}\sum_{\ell=0}^n \frac{x^{r\ell}}{\ell!} \sum_{k=0}^{\ell} \binom{\ell}{k} (rk+\beta+1)_n (-1)^k.
\end{equation}
Hence $L_{\vec{n}}$ is a polynomial of degree $n$ in $x^r$ and it has an $\omega$-symmetry, i.e., for every $k$ we have
\[	L_{\vec{n}}(\omega^k x;\beta) = L_{\vec{n}}(x;\beta),	\]
which implies that the zeros on the rays of the $r$-star are rotated copies of the zeros restricted to the interval $[0,\infty)$.

For the polynomials above the diagonal, we have the following Rodrigues formula.
\begin{theorem}
The type II Laguerre-Angelesco polynomials above the diagonal $\vec{n}+\vec{e}_k$, with $\vec{n}=(n,n,\ldots,n)$ and $\vec{e}_k$ the 
$k^{\mbox{\footnotesize th}}$ unit vector, is given by the Rodrigues formula
\begin{equation}	\label{eq:Rodgrigues_upper}
	x^{\beta} e^{-x^r} L_{\vec{n}+\vec{e}_k}(x;\beta)= \frac{(-1)^n}{r^n} \frac{d^n}{dx^n} [x^{\beta+n} e^{-x^r} L_{\vec{e}_k}(x;\beta+n)],
\end{equation}
where $L_{\vec{e}_k}(x;\beta+n)$ is an ordinary monic orthogonal polynomial of degree $1$ on $[0,\omega^{k-1}\infty)$ with respect to 
$|x|^{\beta+n}e^{-x^r}$. This polynomial is given by
\begin{equation}	\label{eq:upper2}
	L_{\vec{e}_k}(x;\beta+n) = x - \frac{\Gamma(\frac{\beta+n+2}{r})}{\Gamma(\frac{\beta+n+1}{r})}\omega^{k-1}.
\end{equation}
\end{theorem}
\begin{proof}
One can easily check that (\ref{eq:upper2}) is orthogonal to the constant function $1$ with respect to $|x|^{\beta+n}e^{-x^r}$ on 
$[0,\omega^{k-1}\infty)$. Also the degree condition of $L_{\vec{n}+\vec{e}_k}$ is trivially satisfied, and it is monic. 
For the orthogonality we want to check that 
\[	\int_0^{\omega^{j-1}\infty} x^{\ell} |x|^{\beta} e^{-x^r} L_{\vec{n}+\vec{e}_k}(x;\beta)\, dx = 0,\]
for all $0\leq \ell \leq n-1$ if $j \neq k$ and for all $0\leq \ell \leq n$ if $j = k$, with $1\leq k\leq r$. Inserting the right hand side of
(\ref{eq:Rodgrigues_upper}) in the above integral gives
\[	\frac{(-1)^n}{r^n}  \int_0^{\omega^{j-1}\infty} x^{\ell} \frac{d^n}{dx^n} [x^{\beta+n} e^{-x^r} L_{\vec{e}_k}(x;\beta+n)]\, dx. \]
Integrate this  $\ell$ times by parts 
\[	\frac{(-1)^{n+\ell}}{r^n} \ell!   \int_0^{\omega^{j-1}\infty}  
    \frac{d^{n-\ell}}{dx^{n-\ell}} [x^{\beta+n} e^{-x^r} L_{\vec{e}_k}(x;\beta+n)]\, dx,\]
which is zero for all $0\leq \ell \leq n-1$ and all choices of $1\leq j,k\leq r$. In case $j=k$ and $\ell=n$, this integral equals
\[	\frac{n!}{r^n}   \int_0^{\omega^{k-1}\infty} x^{\beta+n} e^{-x^r} L_{\vec{e}_k}(x;\beta+n)\, dx,\]
which is zero due to the orthogonality of $L_{\vec{e}_k}(x;\beta+n)$. 
\end{proof}
A similar formula was obtained in \cite[Chapter 23]{Ismail} for the Jacobi-Angelesco polynomials on $[a,0]\cup[0,1]$ where $a<0$. 

The explicit expression for the type II Laguerre-Angelesco polynomials above the diagonal can be derived from the Rodrigues formula
\begin{multline}	\label{eq:upper}
L_{\vec{n}+\vec{e}_k}(x;\beta)\\
= \frac{(-1)^n}{r^n}  \sum_{m=0}^n \frac{x^{rm}}{m!} \sum_{\ell=0}^m \binom{m}{\ell} (-1)^{\ell} 
\left[ (r\ell+\beta+2)_n x -  \frac{\Gamma(\frac{\beta+n+2}{r})}{\Gamma(\frac{\beta+n+1}{r})}\omega^{k-1}(r\ell+\beta+1)_n  \right].
\end{multline}
Comparing this with (\ref{eq:typeII}) gives
\begin{equation}	\label{eq:upper1}
L_{\vec{n}+\vec{e}_k}(x;\beta)  
= xL_{\vec{n}}(x;\beta+1)- \frac{\Gamma(\frac{\beta+n+2}{r})}{\Gamma(\frac{\beta+n+1}{r})}\omega^{k-1} L_{\vec{n}}(x;\beta).
\end{equation}
Similar as in case $r=1$ (\ref{JA-LA1}) and $r=2$ (\ref{eq:JA-LA2}), the type II Laguerre-Angelesco polynomials $L_{\vec{n}}(x;\beta)$ on the $r$-star (for any multi-index $\vec{n}=(n_1,\ldots,n_r)$) are limiting cases of the type II Jacobi-Angelesco polynomials $P_{\vec{n}}^{(\alpha,\beta)}(x)$ as follows
\[	L_{\vec{n}}(x;\beta) = \lim_{\alpha\rightarrow \infty} \alpha^{\frac{|\vec{n}|}{r}} P_{\vec{n}}^{(\alpha,\beta)}(\alpha^{-1/r}x).\]
Note that this polynomial has degree $rn$ in $x$ and is monic.

\subsubsection{Recurrence relation}   \label{section:rectypeII}
For any $r$, the nearest neighbor recurrence relation \cite{VA} for the type II Laguerre-Angelesco polynomials are given by
\begin{equation}	\label{eq:recrelII}
xL_{\vec{n}}(x;\beta) = L_{\vec{n}+\vec{e}_k}(x;\beta) + b_{\vec{n},k} L_{\vec{n}}(x;\beta) + \sum_{\ell=1}^r a_{\vec{n},\ell}L_{\vec{n}-\vec{e}_{\ell}}(x;\beta),
\end{equation}
for every $1\leq k\leq r$. Note that the recurrence coefficients $a_{\vec{n},\ell}$ are the same coefficients as the ones for the type I Laguerre-Angelesco polynomials. So for $\vec{n}=(n,n,\ldots,n)$ a diagonal multi-index for $r\geq 1$, we have from Proposition \ref{prop:reccoef}
\[	a_{\vec{n},\ell} = \frac{n}{r^2} \frac{\Gamma(\frac{n+\beta+1}{r})}{\Gamma(\frac{n+\beta-1}{r}+1)}\omega^{2(\ell-1)}, \qquad 1\leq \ell \leq r.\]
An explicit expression for the recurrence coefficients $b_{\vec{n},k}$ is given by
\begin{proposition}
Let $\vec{n}=(n,n,\ldots,n)$ be a diagonal multi-index for $r>1$, then
\begin{equation}	\label{eq:reccoeffII}
	b_{\vec{n},k} = \frac{\Gamma(\frac{\beta+n+2}{r})}{\Gamma(\frac{\beta+n+1}{r})}\omega^{k-1}.
\end{equation}
\end{proposition}
\begin{proof}
Since $L_{\vec{n}}(x;\beta)$ is monic, one can compute $b_{\vec{n},k}$ by comparing the coefficients with $x^{rn}$ in (\ref{eq:recrelII}). On the left hand side this coefficient is equal to zero, since $L_{\vec{n}}(x;\beta)$ is a linear combination of $x^{rk}$, with $0 \leq k \leq n$ (\ref{eq:typeII}). Therefore, 
\[	b_{\vec{n},k} = -(\mbox{coefficient with } x^{rn} \mbox{ in } L_{\vec{n}+\vec{e}_k}(x;\beta)).\] 
From (\ref{eq:upper}) (or more easily from (\ref{eq:upper1})), we obtain (\ref{eq:reccoeffII}).  
\end{proof}

\subsubsection{Differential equation}
In this section we give a linear differential equation of order $r+1$ for the type II Laguerre-Angelesco polynomial $L_{\vec{n}}(x;\beta)$. The differential equation is a combination of lowering and raising operators for these polynomials. In the following we use the multi-index $\vec{1}=(1,1,\ldots,1)$ to describe diagonal polynomials.
\begin{lemma}
For $n\geq r+1$, the polynomial $L_{n\vec{1}}(x;\beta)$ given by (\ref{eq:RodriguesLA}) has for $\beta>0$ the raising operator
\begin{equation}	\label{eq:raisingLA}
	(x^{\beta}e^{-x^r}L_{n\vec{1}}(x;\beta))' = -r x^{\beta-1}e^{-x^r}L_{(n+1)\vec{1}}(x;\beta-1),
\end{equation}
and for $\beta>-1$ the lowering operator
\begin{equation}	\label{eq:loweringLA}
	(L_{n\vec{1}}(x;\beta))' = \sum_{k=1}^r b_{k,n} x^{r-1} L_{(n-k)\vec{1}}(x;\beta+k),
\end{equation}
where
\[	b_{k,n} = (-1)^{k+1}\binom{n}{k} \frac{(r-k+1)_{k-1}}{r^{k-2}}.\]
\end{lemma}
\begin{proof}
The raising operator (\ref{eq:raisingLA}) can easily be proven by differentiating the Rodrigues formula (\ref{eq:RodriguesLA}). To prove 
(\ref{eq:loweringLA}) we first observe that by the explicit expression of $L_{n\vec{1}}(x;\beta)$ (\ref{eq:typeII}), 
the polynomial $(L_{n\vec{1}}(x;\beta))'$ is of degree $rn-1$, has leading coefficient $rn$ and can be written as a linear combination of the monomials 
$x^{rk-1}$ where $1 \leq k\leq n$. Therefore, $(L_{n\vec{1}}(x;\beta))'$ is determined by $n-1$ unknowns. Integration by parts shows that for $\beta>0$
\[	\int_0^{\infty} x^k (L_{n\vec{1}}(x;\beta))' x^{\beta+1} e^{-x^r}\, dx 
= - \int_0^{\infty} x^{\beta+k} e^{-x^r} (\beta+k+1-rx^r)L_{n\vec{1}}(x;\beta)\, dx,\]
and by (\ref{eq:orthotypeIILA}) this is zero for $0 \leq k \leq n-r-1$. So $(L_{n\vec{1}}(x;\beta))'$ is a polynomial of $n-1$ unknown coefficients which is orthogonal to all $x^k$ for $0\leq k\leq n-r-1$ with weight $x^{\beta+1} e^{-x^r}$ on the interval $[0,\infty)$. These are $n-r$ orthogonality conditions. The $n$ functions $x^{r-1} L_{(n-j)\vec{1}}(x;\beta+j)$, $1\leq j \leq n$, have the same orthogonality conditions, are of degree at most 
$rn-1$ and are linearly independent. Hence, they span the linear space of polynomials that are a linear combination of $x^{rk-1}$, $1\leq k \leq r$, with the $n-r+1$ orthogonality conditions. Therefore,
\[	(L_{n\vec{1}}(x;\beta))' = \sum_{k=1}^r b_{k,n} x^{r-1} L_{(n-k)\vec{1}}(x;\beta+k).\]
To find the coefficients $b_{k,n}$, $1\leq k\leq r$, one compares the coefficients of $x^{r(n-m)}$, $1\leq m\leq r$, in the latter expansion and the explicit expression of $(L_{n\vec{1}}(x;\beta))'$ that can be derived from (\ref{eq:typeII}).  
\end{proof}
With these operators one can find the differential equation.
\begin{theorem} \label{theorem:diffLAII}
For any $n\in \mathbb{N}$, $r\geq 1$ and $\beta >r-2$ the polynomial $y=L_{n\vec{1}}(x;\beta)$ satisfies the differential equation
\begin{equation}	\label{eq:diffLAII}
	[x^{\beta+1}e^{-x^r}y']^{(r)} + \sum_{k=0}^{r-1} c_{k,n}x^k[x^{\beta}e^{-x^r}y]^{(k)}=0,
\end{equation}
where the coefficients $c_{k,n}$ for $0\leq k \leq r-1$ are given by
\begin{equation}	\label{eq:coeffdiffLAII}
	c_{k,n} = r\binom{r}{k} [(n+k+1)_{r-k}-(k+1)_{r-k}].
\end{equation}
\end{theorem}
\begin{proof}
From the lowering operator (\ref{eq:loweringLA}) one has
\[	x^{\beta+1}e^{-x^r} L_{n\vec{1}}'(x;\beta) = \sum_{j=1}^r b_{j,n} x^{r+\beta}e^{-x^r} L_{(n-j)\vec{1}}(x;\beta+j).\]
Differentiating both sides $r$ times gives
\begin{multline*}
	[x^{\beta+1}e^{-x^r}L_{n\vec{1}}'(x;\beta)]^{(r)} = \\
	\sum_{j=1}^r b_{j,n} \sum_{k=0}^r \binom{r}{k} (x^{r-j})^{(r-k)}[x^{\beta+j} e^{-x^r}L_{(n-j)\vec{1}}(x;\beta+j)]^{(k)}.
\end{multline*}
Now we can use the raising operator 
\begin{multline*}
	[x^{\beta+1}e^{-x^r}L_{n\vec{1}}'(x;\beta)]^{(r)} =\\
	\sum_{j=1}^r b_{j,n} \sum_{k=j}^r \binom{r}{k} (k-j+1)_{r-k} x^{k-j} (-1)^j r^j [x^{\beta} e^{-x^r}L_{n\vec{1}}(x;\beta)]^{(k-j)}.
\end{multline*}
Using the expression of $b_{j,n}$ from the above lemma and changing the order of the two sums gives for $y =L_{n\vec{1}}(x;\beta)$
\[	[x^{\beta+1}e^{-x^r}y']^{(r)} +  \sum_{k=0}^{r-1} c_{k,n} x^k [x^{\beta}e^{-x^r}y]^{(k)}=0,\]
where
\[	c_{k,n} =  r^2\sum_{j=1}^{r-k} \binom{n}{j}\binom{r}{k+j} (r-j+1)_{j-1}(k+1)_{r-k-j}.\]
This is exactly (\ref{eq:diffLAII}) where the coefficients $c_{k,n}$ can be computed to be equal to (\ref{eq:coeffdiffLAII}).
\end{proof}

\subsubsection{Asymptotic zero behavior}
Now we investigate the asymptotic behavior of the zeros of the diagonal type II Laguerre-Angelesco polynomials $L_{\vec{n}}(x;\beta)$, i.e., for the multi-index $\vec{n}=(n,n,\ldots,n)$ and $n\rightarrow\infty$. From the symmetry in the explicit expression of this polynomial (\ref{eq:typeII}), we see that the zeros are of the form
\[	\omega^j x_{1,n}, \omega^j x_{2,n} , \dots, \omega^j x_{n,n}, \qquad 0\leq j \leq r-1,	\]
where $x_{1,n}, x_{2,n} , \ldots, x_{n,n}$ are the real zeros. Similarly as Lemma \ref{lemma1}, one can show that the real zeros of $L_{\vec{n}}(x;\beta)$ are simple and lie in the open interval $(0,\infty)$. 
Also, the largest zero will tend to $\infty$ as $n\rightarrow \infty$. Therefore, to be able to say something about the asymptotic behavior of the zeros, we will rescale the zeros so that they are on the compact interval $[0,1]$. This rescaling will be the same as for the type I polynomials and depends on the behavior of the largest zero.

\begin{lemma}
For $n$ large enough, the real zeros of the polynomials $L_{\vec{n}}$ lie in the interval
\[	V_{n,r}:= \left[0,(r+1)^{1/r}\left(\frac{r+1}{r}\right)^{\frac{r+1}{r}}n^{1/r}\right].	\]
\end{lemma}
\begin{proof}
The proof is similar as for the type I polynomials, see Lemma \ref{lemma:ifri}. Note that in this case, one only needs to consider the $n$ real zeros of $L_{\vec{n}}$.
\end{proof}

Since the previous lemma does not give the smallest possible interval that contains the zeros, we need to determine the smallest possible $\alpha_r$ such that the real zeros of $L_{\vec{n}}$ lie in $[0,\alpha_rn^{1/r}]$. To obtain this optimal $\alpha_r$, we use the rescaling
\[x= \alpha_r n^{\frac{1}{r}} z,\]
such that the asymptotic distribution of the real zeros of $L_{\vec{n}}(\alpha_r n^{1/r} z)$ is supported on the compact interval $[0,1]$, which gives the required $\alpha_r$.

From here on, we take a large enough $N$ such that for all $n\geq N$ all (except possible $o(n)$) of the real zeros of the rescaled $L_{\vec{n}}$ lie in $[0,1]$. Hence, define the polynomial
\[ 	\tilde{L}_{\vec{n}}(x;\beta) := L_{\vec{n}}\left( \alpha_r n^{\frac{1}{r}} x;\beta \right). \]
Then we will analyze the behavior of the real zeros of these new polynomials. Note that from the differential equation of $L_{\vec{n}}(x;\beta)$ in Theorem \ref{theorem:diffLAII}, one can derive the following differential equation for $\tilde{L}_{\vec{n}}$
\begin{equation}
\label{eq:difftildeL}
\left[x^{\beta+1} e^{-\alpha_r^r n x^r}\tilde{L}'_{\vec{n}}(x;\beta)   \right]^{(r)} + \sum_{k=0}^{r-1} C_{k,n} x^k \left[z^{\beta}e^{-\alpha_r^r n x^r }\tilde{L}_{\vec{n}}(x;\beta) \right]^{(k)}=0,  
\end{equation}
where 
\[	C_{k,n} = r\binom{r}{k} \alpha_r^r n [(n+k+1)_{r-k}-(k+1)_{r-k}].	\]

Denote the real zeros of $\tilde{L}_{\vec{n}}(x;\beta)$ by $0< x_{1,n} <  ... < x_{n,n} $ and consider the normalized zero counting measure given by
\[	\mu_n = \frac{1}{rn} \sum_{k=1}^n \sum_{j=1}^r \delta_{\omega^{j-1} x_{k,n}}.\]
Then $(\mu_n)_n$ is a sequence of probability measures on $\Delta_r^*:=\bigcup\limits_{j=1}^r [0,\omega^{j-1}]$, and by Helley's selection principle it will contain a subsequence $(\mu_{n_k})_k$ that converges weakly to a probability measure $\mu$ on $\Delta_r^*$. If this limit is independent of the subsequence, then we call it the asymptotic zero distribution of the zeros of $\tilde{L}_{\vec{n}}(x;\beta)$. We will investigate this by means of the Stieltjes transform
\[	S_n(z) = \int_{\Delta_r^*} \frac{d\mu_n(x)}{z-x} = \frac{1}{rn}\frac{\tilde{L}_{\vec{n}}'(z;\beta)}{\tilde{L}_{\vec{n}}(z;\beta)}, 
      \qquad S(z) = \int_{\Delta_r^*} \frac{d\mu(x)}{z-x}, \qquad z \in \mathbb{C}\setminus \Delta_r^*,  \]
and use the Grommer-Hamburger theorem \cite{Geronimo} which says that $\mu_n$ converges weakly to $\mu$ if and only if $S_n$ converges uniformly on compact subsets of $\mathbb{C}\setminus \Delta_r^*$ to $S$ and $zS \rightarrow 1$ as $z \rightarrow \infty$.

First we show that the weak limit of $(\mu_{n_k})_k$ has a Stieltjes transform $S$ which satisfies an algebraic equation of order $r+1$. 
\begin{proposition}
Suppose that $\mu_{n_k}$ converges weakly to $\mu$, then the Stieltjes transform $S$ of $\mu$ satisfies
\begin{equation}	\label{eq:algSII}
	zr^r(S -\alpha_r^r  z^{r-1})^{r+1} +  \alpha_r^r (1+zrS-\alpha_r^r r z^r)^r=0.
\end{equation}
\end{proposition}
\begin{proof}
Since the Stieltjes transform $S_n$ of $\mu_n$ satisfies
\[	\tilde{L}_{\vec{n}}'(z;\beta) = rn \tilde{L}_{\vec{n}}(z;\beta) S_n(z),\]
one can prove by induction that the higher order derivatives of $\tilde{L}_{\vec{n}}(z;\beta)$ can be described in terms of $S_n$ and its derivatives as follows
\begin{equation}	\label{diffLn}
	\tilde{L}_{\vec{n}}^{(k)}(z;\beta) = r^k n^k \tilde{L}_{\vec{n}}(z;\beta) \left[S_n^k(z) 
    + \frac{1}{n} G_{n,k}[S_n,S_n',\ldots,S_n^{(k-1)}](z)\right], \qquad k\geq 0,
\end{equation}
where $G_{n,k}$ is the same map as for the type I polynomials, given by (\ref{eq:Gfunction}).
Inserting (\ref{diffLn}) in the differential equation of $\tilde{L}_{\vec{n}}$ (\ref{eq:difftildeL}) gives
\begin{eqnarray}  \label{eq:sum}
	0 &=&	 \left[ z^{\beta+1} e^{-\alpha_r^r n z^r} y'  \right]^{(r)} + \sum_{k=0}^{r-1}C_{k,n}z^k 
       \left[ z^{\beta} e^{-\alpha_r^r n z^r}y  \right]^{(k)}    \nonumber \\
	&=&  \sum_{m=0}^{r} \binom{r}{m} y^{(m+1)}[ z^{\beta+1} e^{-\alpha_r^r n z^r}]^{(r-m)} + \sum_{k=0}^{r-1}C_{k,n}z^k  
   \sum_{\ell=0}^k \binom{k}{\ell} y^{(\ell)} [ z^{\beta} e^{-\alpha_r^r n z^r}]^{(k-\ell)}   \nonumber \\
	&=& \sum_{m=0}^{r} \binom{r}{m}r^{m+1}n^{m+1} y [S_n^{m+1}+\mathcal{O}(1/n) ] [z^{\beta+1} e^{-\alpha_r^r n z^r}]^{(r-m)} \nonumber \\
	&& + \sum_{k=0}^{r-1}C_{k,n}z^k  \sum_{\ell=0}^k \binom{k}{\ell} r^{\ell}n^{\ell} y [S_n^{\ell}+\mathcal{O}(1/n) ] 
         [ z^{\beta} e^{-\alpha_r^r n z^r}]^{(k-\ell)} .
\end{eqnarray}
Now we have two terms containing a derivative. To be able to take the limit for $n\rightarrow \infty$, we need to know the behavior of these derivatives in terms of $n$. The derivative in the first sum of (\ref{eq:sum}) equals
\[	[z^{\beta+1} e^{-\alpha_r^r n z^r}]^{(r-m)} = \sum_{i=0}^{r-m} \binom{r-m}{i} (\beta-i+2)_i z^{\beta+1-i} 
         e^{-\alpha_r^r n z^r} \Pi_{r-m-i}(z), 		\]
where $\Pi_{r-m-i}$ is a polynomial in $z$ with leading term $(-\alpha_r^rnrz^{r-1})^{r-m-i}$. This term is also the one with the highest power of $n$ in $\Pi_{r-m-i}$. Therefore, the first sum in (\ref{eq:sum}) behaves as $n^{r+1}$ for $n$ large. This means that for $n \to \infty$ only the term for $i=0$ is relevant. Similarly one can show that the second sum in (\ref{eq:sum}) also behaves like $n^{r+1}$ for large $n$, for this one uses the expression of $C_{k,n}$ to obtain
\[	\lim_{n\rightarrow\infty} C_{k,n} n^{k-r-1} = 	r\binom{r}{k}\alpha_r^r.\]
	
Next, multiply by $\frac{e^{\alpha_r^r n z^r}}{ n^{r+1} z^{\beta} y }$ and let $n=n_k \rightarrow \infty$. Since $S_{n_k}$ converges uniformly to $S$ on compact subsets of $\mathbb{C}\setminus \Delta_r^*$, it follows that all the derivatives $S_{n_k}^{(j)}$ converge uniformly to $S^{(j)}$ on compact subsets of $\mathbb{C}\setminus \Delta_r^*$. In the limit we find
\[	0 = rzS\sum_{m=0}^{r} \binom{r}{m}r^{m} S^{m} (-\alpha_r^r r z^{r-1})^{r-m} 
	+ r\alpha_r^r\sum_{k=0}^{r-1}\binom{r}{k} z^k  \sum_{\ell=0}^k \binom{k}{\ell} r^{\ell} S^{\ell} (-\alpha_r^r r z^{r-1})^{k-\ell} . \]
By using the binomial theorem a few times, one can find equation (\ref{eq:algSII}). Since the equation does not depend on the subsequence $(n_k)_k$ anymore, every converging subsequence has a limit $S$ satisfying (\ref{eq:algSII}). 
\end{proof}

The algebraic equation from the previous proposition is of order $r+1$ so it has $r+1$ solutions, but we are interested in the solution which is a Stieltjes transform of a probability measure on $\Delta_r^*$, which is the solution for which $S(z) = \mathcal{O}(1/z)$.
Then one can find the asymptotic zero distribution measure $\mu$ by using the Stieltjes-Perron inversion formula.
\begin{theorem}
For $\alpha_r = \left(\frac{r+1}{r}\right)^{\frac{r+1}{r}}$, the real zeros of $\tilde{L}_{\vec{n}}(x;\beta)$ lie in $[0,1]$ for $n$ large enough and the asymptotic zero distribution of $\tilde{L}_{\vec{n}}(x;\beta)$ as $n\rightarrow \infty$ is given by a measure which is absolutely continuous on 
$\bigcup\limits_{j=1}^r[0,\omega^{j-1}]$ with a density $u_r$ given by $u_r(x) = x^{r-1}w_r(x^r)$, where $w_r$ is given by
\[	w_r(\hat{x}) = \frac{1}{\pi r \hat{x}}\frac{\sin \theta \sin (r\theta) \sin (r+1)\theta}{|\sin(r+1)\theta e^{i\theta}-\sin r \theta |^2},\]
where we used the change of variables
\[	\hat{x} = x^r = \frac{1}{c_r}\frac{(\sin (r+1)\theta)^{r+1}}{ (\sin r\theta)^r \sin \theta}, \qquad 0 < \theta < \frac{\pi }{r+1},\]	
and $c_r =\frac{(r+1)^{r+1}}{r^r}$.
\end{theorem}
\begin{proof}
The proof is along the same lines as Theorem \ref{theorem:asympI} for the type I polynomials. First we transform the algebraic equation (\ref{eq:algSII}) by taking
\[	W = \frac{1+zrS-\alpha_r^r r z^r}{S -\alpha_r^r  z^{r-1}}, \qquad S = \frac{1-\alpha_r^r r z^r + \alpha_r^rz^{r-1} W}{W-rz},\]
which gives
\begin{equation}	\label{eq:algWII}
	\alpha_r^r W^{r+1} - r\alpha_r^r z W^r + r^r z=0.
\end{equation} 
We look for a solution $W$ of the form $\rho e^{i \theta}$, where $\theta$ is real and $\rho>0$. Take $z=x \in [0,1]$ and insert $W = \rho e^{i\theta}$ into (\ref{eq:algWII}) to find
\[	\alpha_r^r \rho^{r+1} e^{i(r+1)\theta} - r\alpha_r^r x\rho^r e^{ir\theta} + r^r x=0.\]
Hence, the real and imaginary parts satisfy
\begin{eqnarray}
	\label{eq:realII}
	\alpha_r^r \rho^{r+1} \cos(r+1)\theta - r\alpha_r^rx\rho^r \cos r\theta + r^r x&=&0, \\
	\label{eq:imII}
	\alpha_r^r \rho^{r+1} \sin(r+1)\theta - r\alpha_r^r x\rho^r \sin r\theta &=&0.
\end{eqnarray}
From (\ref{eq:imII}) we have
\[	x = \frac{\rho \sin(r+1)\theta}{r \sin r\theta},\]
and using this in (\ref{eq:realII}), we find
\[	\rho({x}) = \frac{r^{1-1/r}\sin^{1/r}(r+1)\theta}{\alpha_r\sin^{1/r} \theta}. \]	
Note that when $\theta\in (0,\pi/(r+1))$ one has $\rho >0$. Combining this gives
\[	x = \frac{1}{r^{1/r}\alpha_r}\frac{\sin^{1+1/r} (r+1)\theta}{ \sin r\theta \sin^{1/r} \theta},\]
or
\[	\hat{x} := x^r = \frac{1}{r\alpha_r^r}\frac{\sin^{r+1} (r+1)\theta}{ \sin^r r\theta \sin \theta},\]
which is the same function as for the type I polynomials (see Theorem \ref{theorem:asympI}). This function gives the zeros in terms of the new 
parameter $\theta$. By differentiation, one can see that it has its maximum in $\theta =0$. By applying de l`H\^{o}pital's rule $r+1$ times, 
one obtains $\hat{x}(0) = \frac{(r+1)^{r+1}}{r^{r+1}\alpha_r^r}$. Therefore, if one takes
\[   	\alpha_r =\left(\frac{r+1}{r}\right)^{\frac{r+1}{r}},  	\]
then for $n \to \infty$ the asymptotic distribution of the real zeros of $\tilde{L}_{\vec{n}}$ is on $[0,1]$, and hence the asymptotic distribution of
all the zeros is on the unit $r$-star $\bigcup\limits_{j=1}^r[0,\omega^{j-1}]$. So for these  values of $x$, equation (\ref{eq:algWII}) has a solution of the form $\rho e^{i\theta}$. Observe that also $\rho e^{-i\theta}$ is a solution and in fact
\[	W_+({x}) = \lim_{\epsilon \rightarrow 0+} W({x}+i\epsilon) = \rho({x}) e^{i\theta} ,
  \qquad 	W_-({x}) = \lim_{\epsilon \rightarrow 0+} W({x}-i\epsilon) = \rho({x}) e^{-i\theta}, \]
are the boundary values of the function $W$ which is analytic on $\mathbb{C} \setminus \bigcup\limits_{j=1}^r[0,\omega^{r-1}]$. From the Stieltjes-Perron inversion formula, we can compute the density $u_r$ as
\[	u_r(x) = \frac{1}{2\pi i} (S_-(x) - S_+(x)) = \frac{\rho}{\pi r} \frac{\sin\theta}{|\rho e^{i\theta}-x|^2}.\]
Writing everything in terms of $\hat{x}$ gives the required result. 
\end{proof}
Note that this probability distribution is defined on the unit $r$-star $\bigcup\limits_{j=1}^r[0,\omega^{j-1}] \subset \mathbb{C}$, so this would require a 3D-plot. For convenience we have plotted in Fig. \ref{figure:AZDII} only the density functions $u_r$ on $[0,1]$ for $r=1,2,3,4,5,6,7$. This density then needs to be copied to each ray of the $r$-star. To get probability densities on $[0,1]$, one needs to multiply the densities by $r$. This would give the same density functions as for the type I polynomials (see Fig. \ref{figure:AZD}). Therefore, the order of the singularities will be the same as for the type I polynomials.  
This shows again the typical behavior of zeros for an Angelesco system, where the zeros on all the other intervals $[0,\omega^{j-1}]$, $j=2,3,...,r$ 
push the zeros on $[0,1]$ to the right.
\begin{figure}[t]
\centering 
\includegraphics[scale=0.35]{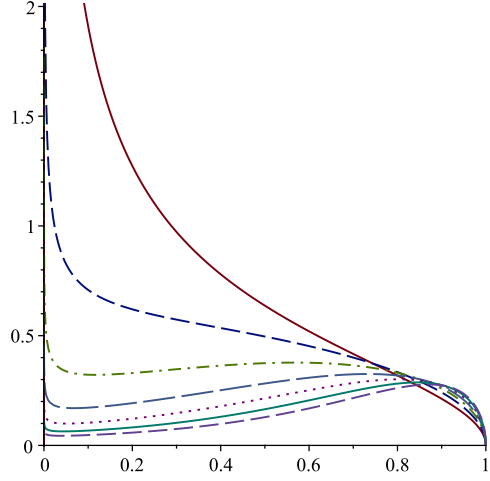}
\caption{The density $u_r$ of the asymptotic zero distribution on $[0,1]$ for $r=1$ (solid in red), $r=2$ (dash in blue), $r=3$ (dash-dot), 
$r=4$ (long dash), $r=5$ (dots), $r=6$ (solid in green) and $r=7$ (dash in purple).}
\label{figure:AZDII}
\end{figure}

\section*{Acknowledgements}
This work was supported by FWO research project G.086416N and EOS project PRIMA 30889451.

\end{document}